\documentclass[11pt]{amsart}
\usepackage{amsmath, enumerate}
\usepackage{amsfonts, amssymb, amsthm}
\usepackage{mathtools}
\usepackage{geometry}
\usepackage{a4wide}
\usepackage{listings}
\usepackage{xcolor}
\usepackage{bbm}
\usepackage{mathrsfs}
\usepackage{layout}
\usepackage[mathcal]{euscript}
\usepackage{anysize}
\marginsize{22mm}{18mm}{17mm}{30mm}
\usepackage{todonotes,blindtext}
\usepackage{bm}
\usepackage{amsmath,amssymb,amsthm,amscd,amsfonts,comment}
\usepackage{bbm}
\usepackage{todonotes}
\usepackage{graphics,graphicx}
\usepackage{mathtools}
\usepackage{mathrsfs}
\usepackage{hyperref}
\usepackage{amsxtra}
\usepackage[square]{natbib}
\usepackage{caption}
\usepackage{tcolorbox}
\usepackage[all]{xy}
\hypersetup{
    colorlinks=true,
   citecolor=blue,
    filecolor=black,
    linkcolor=red,
    urlcolor=black
}
\bibpunct{[}{]}{,}{n}{}{;}

\numberwithin{equation}{section}
\theoremstyle{plain}
\newtheorem{thm}{Theorem}[section]

\newtheorem{lem}[thm]{Lemma}
\newtheorem{prop}[thm]{Proposition}

\theoremstyle{definition}
 
\theoremstyle{definition}
\newtheorem{rem}[thm]{Remark}

\let\Im\relax
\DeclareMathOperator{\Im}{Im}

\let\ord\relax
\DeclareMathOperator{\ord}{ord}

\DeclareMathOperator{\del}{\Delta_{\mathrm{hyp},\textit{z}}}

\DeclareMathOperator{\hyp}{\mu_{hyp}} 
\DeclareMathOperator{\PSL}{\mathrm{PSL}_{2}(\mathbb{R})}
\DeclareMathOperator{\PSLZ}{\mathrm{PSL}_{2}(\mathbb{Z})}
\DeclareMathOperator{\can}{\mu_{can}}

\DeclareMathOperator{\shyp}{\mu_{shyp}}

\let\Re\relax
\DeclareMathOperator{\Re}{Re}
\let\id\relax
\DeclareMathOperator{\id}{\mathrm{id}}

\DeclareMathAlphabet{\mathpzc}{OT1}{pzc}{m}{it}
\makeatletter
\newcommand*{\rom}[1]{\expandafter\@slowromancap\romannumeral #1@}
\makeatother
\makeatletter

\newcommand{\Rmnum}[1]{\expandafter\@slowromancap\romannumeral #1@}
\makeatother
\begin{document}
\makeatletter
\def\imod#1{\allowbreak\mkern10mu({\operator@font mod}\,\,#1)}
\makeatother
\title{Bounds for canonical Green's functions at cusps}
\author{Priyanka Majumder}
\address{Priyanka Majumder, Theoretical Statistics and Mathematics Unit, Indian Statistical Institute, Bangalore Centre, 8th Mile, Mysore Road, 560059 Bangalore, India}
\email{pmpriyanka57@gmail.com}

\author{Anna-Maria von Pippich}
\address{Anna-Maria von Pippich, Fachbereich Mathematik und Statistik, Universit\"at Konstanz, Universit\"atsstra{\ss}e 10, 78464 Konstanz, Germany}
\email{anna.pippich@uni-konstanz.de}

\begin{abstract} 

Let $\Gamma$ denote a cofinite Fuchsian subgroup. In the context of Arakelov theory, the canonical Green's function associated with $\Gamma$ plays a crucial role in establishing asymptotic behavior for Arakelov invariants of the modular curve related to a congruence subgroup of level $N$, where $N$ is a positive integer. More precisely, the canonical Green's functions evaluated at certain cusps contribute to the analytic component of the asymptotic formula for the self-intersection of the relative dualizing sheaf. This article presents a proof demonstrating that the canonical Green's function of a cofinite Fuchsian subgroup, evaluated at cusps, is bounded by the scattering constants, Kronecker's limit functions, and the Selberg zeta function associated with the group $\Gamma$. As an application, we establish an asymptotic expression for the canonical Green's function linked to $\Gamma_0(N)$, where $N$ is any positive integer.

\end{abstract}

\maketitle

\section{Introduction}
\subsection{Overview}\label{sec_overview}
Let $\Gamma\subset \PSL$ be a cofinite Fuchsian subgroup acting by fractional linear transformations on the upper half-plane $\mathbb{H}$. By $\mathcal{P}_\Gamma$ resp.~$\mathcal{E}_\Gamma$ we denote a complete set of inequivalent cusps and elliptic fixed points of $\Gamma$, respectively,
and we set $p_\Gamma:= \sharp \mathcal{P}_\Gamma$, $e_\Gamma:= \sharp \mathcal{E}_\Gamma$. 
The quotient space $Y = \Gamma \backslash \mathbb{H}$ has the structure  of non-compact Riemann surface of genus $g_Y$, having 
$p_\Gamma$ cusps and $e_\Gamma$ elliptic fixed points. The compactification $X:=\overline{Y}=Y \cup \mathcal{P}_\Gamma$ inherits the structure of a compact Riemann surface of genus $g_Y$.  
We consider the hyperbolic metric, locally, for $z\in Y\setminus \mathcal{E}_\Gamma$, given by
\begin{align*}
\hyp(z) = \frac{i}{2} \cdot \frac{dz \wedge d\overline{z}}{{\Im(z)}^2}.
\end{align*}

The hyperbolic metric is singular at the cusps and the elliptic fixed points of $Y$.
The volume $\mathrm{vol}_{\mathrm{hyp}}(Y)$ of $Y$ with respect to the hyperbolic metric $\hyp$ is finite and will be denoted by $v_Y$.
The rescaled hyperbolic metric $\shyp(z):=\hyp(z)/v_Y$ measures the volume of $Y$ to be one.
Additionally, on $Y$, we consider the canonical metric
\begin{align*}
\can(z) = \frac{i}{2g_Y} \sum_{j=1}^{g_Y} \left| f_j(z) \right|^2 dz \wedge d\overline{z},
\end{align*}
where $\{ f_1, \dots, f_{g_Y} \}$ denotes an orthonormal basis of the space of cusp forms of weight 2 with respect to $\Gamma$, endowed with the Petersson inner product. 
The canonical metric $\can$ extends smoothly to the cusps and the elliptic fixed points of $Y$ yielding the canonical metric $\can$ on $X$, which is a smooth metric on $X$. 
For $z, w\in X$ with $z\not=w$, the Green's function $\mathcal{G}_{\mathrm{can}}(z,w)$ associated with the metric $\can$ is
given as the unique solution to the differential equation
\begin{align*}
d_z d_z^c  \mathcal{G}_{\mathrm{can}}(z,w) + \delta_w(z) = \can(z),
\end{align*}
where $\delta_w(z)$ is the Dirac delta distribution, with normalization condition
\begin{align*}
\int\limits_X \mathcal{G}_{\mathrm{can}}(z,w)(z,w)\can(z) = 0,
\end{align*}
for all $w\in X$. The function $\mathcal{G}_{\mathrm{can}}(z,w)$ is referred to as the canonical Green's function on $X$; it is a function on $X\times X$ admitting a logarithmic singularity along the diagonal.
Aryasomayajula \cite{A12} showed that the restriction of $\mathcal{G}_{\mathrm{can}}(z,w)$ to $Y\times Y$ coincides with the canonical Green's function on $Y$.\\
In this article, we establish an upper bound for the canonical Green's function $\mathcal{G}_{\mathrm{can}}(z,w)$ evaluated at two distinct cusps of $X$. 
To achieve this, we consider a Green's function better suited to the quotient structure of $Y=\Gamma \backslash \mathbb{H}$, namely the 
(rescaled) hyperbolic Green's function $\mathcal{G}_{\mathrm{hyp}}(z,w)$, which is associated with the rescaled hyperbolic metric $\shyp$; see subsection 2.4 for its precise definition.
Unlike the canonical Green's function, which is smooth at the cusps, the hyperbolic Green's function $\mathcal{G}_{\mathrm{hyp}}(z,w)$ exhibits a log-log-singularity at the cusps.\\
Building upon the work of Jorgenson and Kramer \cite{JK06}, Aryasomayajula \cite{A15Z}, \cite{A12}
expresses the difference $\mathcal{G}_{\mathrm{can}}-\mathcal{G}_{\mathrm{hyp}}$ in terms of integrals involving the hyperbolic heat kernel $K_{\mathrm{hyp}}(t;z,w)$ ($t > 0$, $z, w \in Y$). Relying essentially on these results, we derive an explicit bound for the canonical Green's function evaluated at two distinct cusps. This bound involves spectral and arithmetic quantities of $Y$ such as scattering constants, Kronecker's limit functions, the Selberg zeta constant, and the smallest nonzero eigenvalue of the hyperbolic Laplacian.

As an application, we examine the congruence subgroup $\Gamma_0(N)$ for a positive integer $N$ and refine the bound for the canonical Green's function in this setting.  Using results from \cite{CI00} on scattering constants and a simplified expression for the Kronecker limit function for $\Gamma_0(N)$, we apply the bound on the Selberg zeta function from \cite{JK01}. Finally, by utilizing the spectral bound $\lambda_1 \geq 21/100$ for the smallest nonzero eigenvalue of the Laplacian on congruence subgroups (Luo, Rudnick, Sarnak \cite{LRS95}), we derive asymptotics for the canonical Green's function associated with $\Gamma_0(N)$. 

\subsection{Applications}
Given a smooth algebraic curve defined over a number field, along with its minimal regular model over the corresponding ring of integers, Arakelov introduced a real number, known as the self-intersection of the dualizing sheaf, in \cite{arakelov}. The study of this Arakelov invariant is motivated by its relevance in arithmetic geometry. According to \cite{arakelov}, the Arakelov self-intersection of the dualizing sheaf on a modular curve is defined as the sum of a geometric part and an analytic part, with the analytic part computed using the values of the canonical Green’s function at cusps.

 Considering the congruence subgroup $\Gamma=\Gamma_0(N)$ and $Y_0(N)=\Gamma\backslash \mathbb{H}$ with compactification $X_0(N)=\overline{\Gamma\backslash \mathbb{H}}$, with a positive squarefree integer $N$ with $2, 3 \nmid N$, Abbes--Michel--Ullmo \cite{abbes}, \cite{MU}, proved the following asymptotics:
\begin{align*}
    \bar{\omega}_{\mathcal{X}_0(N)}^2=3 g_{Y_0(N)}\log N+o( g_{Y_0(N)}\log N)\,\ \text{as}\,\ N\to \infty,
\end{align*}
where $\bar{\omega}_{\mathcal{X}_0(N)}^2$ is the self-intersection of the dualizing sheaf of the minimal regular model $\mathcal{X}_0(N)$ over $\mathbb{Z}$ for the modular curve $X_0(N)$. In their case, they proved that the leading term, $3 g_{Y_0(N)}\log N$, is the sum of $g_{Y_0(N)}\log N$ from the geometric part and $2 g_{Y_0(N)}\log N$ from the analytic part. The analytic part, in this case, is exactly given by
\begin{align*}
2g_{Y_0(N)} (1-g_{Y_0(N)}) \, \mathcal{G}_{\mathrm{can}}(0, \infty).
\end{align*}
In the present article we remove the squarefree condition on $N$, and prove that for a positive integer $N$, the leading term in the asymptotics for $2g_{Y_0(N)} (1-g_{Y_0(N)}) \, \mathcal{G}_{\mathrm{can}}(0, \infty)$ is $2 g_{Y_0(N)}\log N$.

For a squarefree natural number $N$ such that $2, 3 \nmid N$, Abbes--Michel--Ullmo \cite{abbes}, \cite{MU}, proved an asymptotic expression for the canonical Green's function on the compactification $X$ of $Y=\Gamma\backslash \mathbb{H}$  by using the following formula:
\begin{align}\label{abbesulmoformula}
\mathcal{G}_{\mathrm{can}}(p_k,p_l)&=4\pi\mathcal{C}_{p_kp_l}+\frac{4\pi}{v_Y}+\lim_{s\to 1}\left(\,\,\int\limits_{{Y}\times {Y}}\,\mathcal{G}_{\mathrm{hyp},\, \hspace{-.03cm} s}(z,w)\can(z)\can(w)-\frac{4\pi}{s(s-1)v_Y}\right)\notag\\
&\phantom{=}- 4\pi\lim_{s\to 1}\left(\,\int\limits_{{Y}} E_{p_k}(z,s)\can(z)+\int\limits_{{Y}} E_{p_l}(w,s)\can(w)-\frac{2}{(s-1)v_Y}\right),
\end{align}
where $p_k,p_l$ are two different cusps of a cofinite Fuchsian subgroup $\Gamma$, and $\mathcal{C}_{p_kp_l}$ is the scattering constant with respect to the cusps $p_k$ and $p_l$. Here, by $\mathcal{G}_{\mathrm{hyp},\, \hspace{-.03cm} s}(z,w)$ we denote the automorphic Green's function, and  by $E_{p_k}(z,s)$ we denote the Eisenstein series corresponding to the cusp $p_k$. 
Using \eqref{abbesulmoformula}, Mayer \cite{Mayer}, investigated the case $X_1(N)$ with an odd squarefree integer $N$ which is divisible by at least two relatively prime numbers bigger than or equal to $4$. In this case, Mayer proved an asymptotic expression for the Arakelov self-intersection of the dualizing sheaf on $X_1(N)$. Recently, Grados--von Pippich \cite{Grados-Pippich} (also see \cite{Grados}), by following the line of proof in \cite{abbes}, proved an asymptotic expression of this quantity for the modular curve $X(N)$ with an odd squarefree positive integer $N$. In \cite{DDC}, Banerjee--Borah--Chaudhari investigated the case of the modular curve $X_0(p^2)$ with a prime number $p$ by following mostly the lines
of proof in \cite{abbes}.

In this article, we propose a different approach to compute the asymptotics of the canonical Green’s function at two different cusps. Instead of using the well-known formula \eqref{abbesulmoformula}, we rely on a formula expressing 
$\mathcal{G}_{\mathrm{can}}-\mathcal{G}_{\mathrm{hyp}}$ in terms of heat kernel integrals, as outlined in subsection \ref{sec_overview}.
This method enables us to avoid the elaborate Rankin–Selberg calculations that are usually necessary in this context.

Additionally, we establish an upper bound for the canonical Green's function applicable to any cofinite Fuchsian subgroup of $\PSL$. Consequently, our result can be employed for various congruence subgroups such as $\Gamma_0(N)$, $\Gamma_1(N)$, or $\Gamma(N)$, where $N$ is a positive integer. These bounds hold potential for applications in arithmetic algebraic geometry. We also wish to highlight recent progress, notably in \cite{DM23},
where the authors have leveraged the results presented in this article to obtain an asymptotic expression for the Arakelov self-intersection number of the relative dualizing sheaf in Edixhoven's minimal regular model for specific modular curves.

\subsection{Main results}
We now state the main results of the article.
\begin{thm}\label{maintheorem1article}
Let $\Gamma$ be a cofinite Fuchsian subgroup and $Y=\Gamma\backslash \mathbb{H}$ with compactification $X$. For two different cusps $p_k$, $p_l$ of $X$, we have
\begin{align*}
\mathcal{G}_{\mathrm{can}}(p_k,p_l)&= 4 \pi\, \mathcal{C}_{p_k p_l}+ \frac{2\pi}{g_Y} \sum_{\substack{j=1 \\ j\not= k}}^{p_\Gamma}\mathcal{C}_{p_k p_j}+\frac{2\pi}{g_Y} \sum_{\substack{j=1 \\ j\not= l}}^{p_\Gamma}\mathcal{C}_{p_l p_j}+\frac{4\pi c_Y}{g_Y v_Y}\\
&+\frac{2\pi}{g_Y}\sum_{j=1}^{e_\Gamma}  \left( 1- \frac{1}{\ord(e_j)}\right) \big(\mathcal{K}_{p_k}(e_j)
\phantom{=}+\mathcal{K}_{p_l}(e_j)\big)+\delta_Y,
\end{align*}
where $\mathcal{C}_{p_k p_l}$ denotes the scattering constant with respect to cusps $p_k$ and $p_l$,  $\mathcal{K}_{p_k}(e_j)$  denotes the Kronecker's limit function with respect to the cusp $p_k$ evaluated at the elliptic fixed point $e_j$ of order $\ord(e_j)$, $c_Y$ denotes a constant related to the Selberg zeta function on $Y$ (see \eqref{selbergzetaconstant}), and the absolute value of $\delta_Y$ is bounded by
\begin{align*}
&\frac{4 \pi}{v_Y g_Y}\sum_{j=1}^{e_\Gamma} \left( 1+ \frac{1}{\ord(e_j)}\right)+\frac{4\log 2}{v_Y g_Y}\sum_{j=1}^{e_\Gamma}\big( {\ord(e_j)}+1\big)+\frac{4\pi(d_Y+1)^2}{\lambda_1 v_Y}\\&+\frac{4\pi}{v_Y}+\frac{2\log(4 \pi)}{g_Y}
 +\frac{2 p_\Gamma}{g_Y v_Y}\bigg(\pi+\frac{4\pi^2}{3}+1\bigg).
\end{align*}
Here, 
by $\lambda_1$ we denote the smallest non-zero eigenvalue of the hyperbolic Laplacian $\del$ and $\displaystyle d_{Y}=\sup_{z\in Y}\left|\frac{\can(z)}{\shyp(z)}\right|$. 
\end{thm}

We use Theorem \ref{maintheorem1article} to establish the following result for the special case of the modular curves $X_0(N)$.

\begin{thm}\label{maintheorem2article}
Let $\Gamma=\Gamma_0(N)$ and $Y=\Gamma_0(N)\backslash \mathbb{H}$ with compactification $X_0(N)=\overline{\Gamma\backslash \mathbb{H}}$, where $N$ is a positive integer. 
Then, for the cusps $0$ and $\infty$ of $X_0(N)$, we have the following asymptotic expression
\begin{align*}
2g_Y(1-g_Y)\, \mathcal{G}_{\mathrm{can}}(0, \infty)=2g_Y \log N + o(g_Y \log N) \,\ \text{as}\,\ N \to \infty,
\end{align*}
where $g_Y$ denotes the genus of $Y$.
\end{thm}

\subsection{Outline of the article}
The paper is organized as follows. In section \ref{BM}, we recall and summarize basic notation and definitions used in this article.
In section \ref{section5}, we prove Theorem \ref{maintheorem1article}. Then as an application we consider the congruence subgroup $\Gamma_0(N)$. In section \ref{BSCGN}, we prove asymptotic bounds for scattering constants of $\Gamma_0(N)$. In section \ref{BKLFGN}, we prove asymptotic bounds for Kronecker's limit functions of $\Gamma_0(N)$. Finally, in section \ref{section8}, we prove Theorem \ref{maintheorem2article}.
\subsection{Acknowledgements}The authors would like to acknowledge support from the LOEWE research unit ``Uniformized structures in Arithmetic and Geometry" of Technical University of Darmstadt and Goethe University of Frankfurt. The
second named author would also like to acknowledge support from Indian Statistical Institute Bangalore. The authors would also like to thank Anilatmaja Aryasomayajula, Jan Hendrik Bruinier, and J\"urg Kramer for their valuable contributions and insightful mathematical discussions.
They also thank the anonymous referee for thoughtful comments and suggestions that helped to improve the clarity and precision of the presentation.

\section{Background material}\label{BM}
\subsection{Basic notation}

As mentioned in the introduction, we let $\Gamma \subset \PSL$ denote a
cofinite Fuchsian subgroup acting by fractional linear transformations on the hyperbolic
upper half-plane $\mathbb{H}:= \lbrace z=x+iy \in \mathbb{C} \mid x, y \in \mathbb{R}; y>0 \rbrace$. The quotient space $Y:= \Gamma \backslash \mathbb{H}$ admits the structure of a finite volume hyperbolic Riemann surface of genus $g_Y$. \\
By $\mathcal{P}_\Gamma$ resp.~$\mathcal{E}_\Gamma$ we denote a complete set of inequivalent cusps and elliptic fixed points of $\Gamma$, respectively,
and we set $p_\Gamma:= \sharp \mathcal{P}_\Gamma$, $e_\Gamma:= \sharp \mathcal{E}_\Gamma$, that is, we assume that $Y$ has $e_\Gamma$ elliptic fixed points
and $p_\Gamma$ cusps. 
The stabilizer group $\Gamma_{e_j}:= \lbrace \gamma \in \Gamma \mid \gamma e_j= e_j \rbrace$ of the elliptic fixed point $e_j\in\mathcal{E}_\Gamma$ ($j=1, \ldots, e_\Gamma$) is a finite cyclic group of order $\ord(e_j)$, where $\ord(e_j)\in \mathbb{N}_{\geq 2}$ denotes the order of $e_j$. The stabilizer group $\Gamma_{p_j}:= \lbrace \gamma \in \Gamma \mid \gamma p_j= p_j \rbrace$ of a cusps $p_j\in\mathcal{P}_\Gamma$ ($j=1, \ldots, p_\Gamma$) is an infinite group satisfying $\Gamma_{p_j}\simeq\mathbb{Z}$.
By $\mathcal {H}_{\Gamma}$ we denote a complete set of representatives of inconjugate, primitive, hyperbolic elements of $\Gamma$. 
By $X$ we denote the compactified Riemann surface $\overline{Y}=Y\cup \mathcal{P}_{\Gamma}$, obtained from $Y$ by adding
the $p_{\Gamma}$ cusps of $Y$; the elements of $\mathcal{P}_{\Gamma}$ are also called cusps of $X$.\\
We identify $Y$ locally with its universal cover $\mathbb{H}$. By $ds_{\mathrm{hyp}}^2$  we denote the
hyperbolic line element and by $\hyp$ the hyperbolic metric on $Y$. For $z=x+iy\in Y$, we have
\begin{align*}
    ds_{\mathrm{hyp}}^2(z)=\frac{dz\cdot d\overline{z}}{\Im(z)^2}=\frac{dx^2+dy^2}{y^2}, \quad \hyp(z)=\frac{i}{2}\cdot\frac{dz\wedge d\overline{z}}{{\Im(z)}^{2}}=\frac{dx dy}{y^{2}}.
\end{align*}
Since $\Gamma$ is cofinite, the hyperbolic volume $$v_Y:=\int\limits_Y \hyp$$ of $Y$ is finite. By $\mathcal{F}_\Gamma$ we denote a
fundamental domain for $\Gamma$, which is a connected domain of $\mathbb{H}$ which represents 
$Y$. Then we have
\begin{align*}
v_Y=\int\limits\limits_{\mathcal{F}_\Gamma}\hyp(z).
\end{align*}
We defined the rescaled hyperbolic metric as
\begin{align*}
    \shyp:= \frac{\hyp}{v_Y}.
\end{align*}
The hyperbolic Laplacian $\del$ on $Y$ is given, for $z = x + iy \in Y$, by
\begin{align*}
    \del:= -y^{2}\bigg(\frac{\partial^{2}}{\partial x^{2}} +
 \frac{\partial^{2}}{\partial y^{2}}\bigg).
\end{align*}
We recall that, for any smooth function $f$ on $Y$, we have (see, e.g., \cite{Lang88}, p.~10) 
\begin{align*}
\del f(z) \hyp(z)= -4\pi d_z d_z^cf(z)
\end{align*}
with the differential operators $d_z:=\left(\partial_z + \overline{\partial}_z \right)$ and $d_z^{c}:= \left( \partial_z - \overline{\partial}_z
\right)\slash 4\pi i$ on $Y$.
By $d_{\mathrm{hyp}}(z,w)$ we denote the hyperbolic distance between $z, w\in\mathbb{H}$ derived from $ds^{2}_{\mathrm{hyp}}$. Then, we have following relation (see, e.g., \cite{Bea}, Theorem 7.2.1)
\begin{align}\label{cosh-hypdist}
\cosh\bigl(d_{\mathrm{hyp}}(z,w)\bigr)= 1+2 u(z,w)
\end{align}
with the point-pair invariant
\begin{align}\label{defn_u}
u(z,w):=\frac{\left|z-w\right|^{2}}{4\,\mathrm{Im}(z)\mathrm{Im}(w)}\,.
\end{align}
Finally, we recall the following well-known formula (see, e.g., \cite{Sh}, Theorem 2.20)
\begin{align}\label{volumegenusformula}
\frac{v_Y}{2\pi}=2 g_Y-2+p_\Gamma+\sum_{j=1}^{e_\Gamma}\Big(1-\frac{1}{\ord(e_j)}\Big).
\end{align}

\subsection{Kronecker's limit functions and scattering constants }\label{Eis}
In this subsection we recall some results for the classical non-holomorphic parabolic Eisenstein series.
For more details, we refer the reader to the literature, e.g., \cite{hej}, \cite{Kub}, or \cite{I02}.
\\
For $z\in\mathbb{H}$ and $s\in\mathbb{C}$ with $\Re(s)> 1$, the Eisenstein 
series $E_{p_j }(z,s)$ associated to the cusp $p_j\in\mathcal{P}_\Gamma$ is defined 
by the series
\begin{equation}\label{Eisensteinseries}
{E}_{p_j}(z,s) = \sum_{\gamma \in \Gamma_{p_j}\backslash \Gamma}
\Im(\sigma_{p_j}^{-1}\gamma z)^{s},
\end{equation}
where $\sigma_{p_j}\in\mathrm{PSL}_2(\mathbb{R})$ is a scaling matrix of the cusp $p_j$, i.e., it satisfies $$\sigma_{\small{p_j}}\infty = p_j \quad\text{ and }\quad
\sigma_{\small{p_j}}^{-1}\Gamma_{p_j}\sigma_{\small{p_j}}= \langle\left(\begin{smallmatrix} 1 & 1\\ 0 & 1  \end{smallmatrix}\right)\rangle.$$
Note that this definition does not depend on the choice of the representative for the cusp $p_j$, nor on the choice of the scaling matrices. 
The series \eqref{Eisensteinseries} converges absolutely and locally uniformly for any $z\in\mathbb{H}$ and $s\in\mathbb{C}$ 
with $\Re(s)>1$. Thus, it is a holomorphic function for $s\in\mathbb{C}$ 
with $\Re(s)>1$, and it is invariant under the action of $\Gamma$.
Moreover, it
satisfies the differential equation
\begin{align*}
\big(\del -s(1-s)\big){E}_{p_j}(z,s)=0.
\end{align*}
The Eisenstein series $E_{p_j }(z,s)$ admits a meromorphic continuation to all $s\in\mathbb{C}$ with a simple pole 
at $s = 1$ with residue ${1}\slash {v_Y}$. 
For $z\in\mathbb{H}$, the Kronecker's limit function 
$\mathcal{K}_{p_j}(z)$ associated to the cusp $p_j\in\mathcal{P}_{\Gamma}$ is then defined by
\begin{align}\label{KLFdefn}
\mathcal{K}_{p_j}(z)= \lim_{s \to 1}\bigg({E}_{p_j }(z,s)- \frac{1}{(s-1)v_Y}\bigg).
\end{align}
For example, for $\Gamma=\PSLZ$ and $Y=\PSLZ\backslash \mathbb{H}$, we can choose $P_{\Gamma}=\{p_1\}=\{\infty\}$, $\sigma_{\infty}=\mathrm{id}$, and we have $v_Y=\pi/3$. The Kronecker's limit formula
then states that 
\begin{align*}
&\mathcal{K}_\infty(z)=-\frac{1}{2\pi}\log(|\Delta(z)|\Im(z)^{6} )+ \mathcal{C},
\end{align*}
with
\begin{align}\label{scatteringconstantPSLZ}
\mathcal{C}:=\frac{6}{\pi}\left(1-\log (4\pi)-12 \zeta'(-1)\right).
\end{align}
Here, $\Delta(z)$ denotes the classical modular discriminant, a modular form of weight $12$ for 
$\PSLZ$, with Fourier expansion of the form
\begin{align*}
\Delta(z)&= \sum_{n=1}^\infty \tau(n)e^{2\pi i n z},
\end{align*} 
where $\tau(n)$ denotes Ramanujan's tau function. 

\vspace{.1cm}\noindent
Let $p_k,p_l\in \mathcal{P}_\Gamma$ be cusps. Then, for $z\in\mathbb{H}$ and $s\in\mathbb{C}$ with $\Re(s)> 1$, the Fourier expansion
of $E_{p_k }(z,s)$ with respect to the cusp $p_l$ is given by
\begin{align}\label{FourierEisenstein}
{E}_{p_k}(\sigma_{p_l}z,s)
&= \delta_{p_kp_l}\Im(z)^{s} + \varphi_{p_kp_l}(s)\Im(z)^{1-s}\notag\\
&\phantom{=}+ \sum_{n\not=0}\varphi_{p_kp_l}(n,s)\Im(z)^{1/2}K_{s-1/2}(2\pi|n|\Im(z))\,e^{2\pi i n\Re(z)}.
\end{align}
Here, by $\delta_{p_kp_l}$ we denote the Kronecker delta symbol and 
$K_{\mu}(z)$ denotes the modified Bessel function of the second kind (see, e.g., \cite{I02}, Appendix B.4), and we have set
\begin{align}
\varphi_{p_kp_l}(s)&:= \sqrt{\pi}\, \frac{\Gamma(s-1\slash 2)}{\Gamma(s)}\sum_{c>0} {c^{-2s}}\,\mathcal{S}_{p_kp_l}(0,0;c),\label{def_scatteringfct} \\
\varphi_{p_kp_l}(n,s)&:=\frac{2\pi^s}{\Gamma(s)}|n|^{s-1/2}\sum_{c>0} {c^{-2s}}\,\mathcal{S}_{p_kp_l}(0,n;c).\label{phin}
\end{align}
The function $\varphi_{p_kp_l}(s)$, defined in \eqref{def_scatteringfct}, is called scattering function of $\Gamma$ at the cusps $p_k$ and $p_l$. Note that the definition of $\varphi_{p_kp_l}(s)$ does not depend on the choice of the representatives for the cusps $p_k$ and $p_l$. The scattering function $\varphi_{p_kp_l}(s)$ is holomorphic for $s\in\mathbb{C}$ with $\mathrm{Re}(s)>1$, and admits a meromorphic continuation to the whole complex $s$-plane. At $s=1$ there is always a simple pole of $\varphi_{p_kp_l}(s)$ with residue equal to ${1}\slash {v_Y}$. Furthermore, we have $\varphi_{p_kp_l}(s)=\varphi_{p_lp_k}(s)$.

The scattering constant $\mathcal{C}_{p_kp_l}$ of $\Gamma$ at the cusps $p_k, p_l$ is defined by
\begin{align}\label{scatteringconstantdef}
\mathcal{C}_{p_kp_l}:=\lim_{s\to 1}\left(\varphi_{p_kp_l}(s)-\frac{1}{(s-1)v_Y}\right).
\end{align}
By the definition, $\mathcal{C}_{p_kp_l}$ also does not depend on the choice of the representatives for the cusps $p_k$ and $p_l$, and we note the identity
$$\mathcal{C}_{p_kp_l}=\mathcal{C}_{p_lp_k}.$$ 
The function $\varphi_{p_kp_l}(n,s)$, defined in \eqref{phin}, is holomorphic for $s\in\mathbb{C}$ with $\mathrm{Re}(s)>1$, it admits a meromorphic continuation to the whole complex $s$-plane
and it is holomorphic at $s=1$.

For $\Gamma=\PSLZ$, the scattering function $\varphi(s):=\varphi_{\infty\infty}(s)$ 
and the function $\varphi(n,s):=\varphi_{\infty\infty}(n,s)$ are 
explicitly given by 
\begin{align*}
\varphi(s)&=\sqrt{\pi} \,\frac{\Gamma(s-\frac{1}{2})}{\Gamma(s)} \frac{\zeta(2s-1)}{\zeta(2s)},\\
 \varphi(n,s)&=\frac{2\pi^s}{\Gamma(s)\zeta(2s)}|n|^{s-1/2}
 \sum_{\substack{d|n\\d>0}}d^{1-2s},
\end{align*}
and the scattering constant $\mathcal{C}_{\infty\infty}$ is given by the constant arising in \eqref{scatteringconstantPSLZ}, namely
\begin{align*}
\mathcal{C}_{\infty\infty}=\mathcal{C}=\frac{6}{\pi}\left(1-\log (4\pi)-12 \zeta'(-1)\right).
\end{align*}
\subsection{Hyperbolic heat kernels and Selberg zeta constant}
The hyperbolic heat kernel $K_{\mathbb{H}}(t;z,w)$ for $t\in\mathbb{R}_{> 0}$ and $z, w\in \mathbb{H}$ is given by the following formula (see, e.g., \cite{Ch}, p.~246)
\begin{equation*}
K_{\mathbb{H}}(t;z,w)= \frac{\sqrt{2}e^{- t\slash 4}}{(4\pi t)^{3\slash 2}}
\int\limits_{d_{\mathrm{hyp}}(z,w)}^{\infty}\frac{re^{-r^{2}\slash 4t}}{\sqrt{\cosh(r)-\cosh (d_{\mathrm{hyp}}(z,w))}}\,dr.
\end{equation*}
Note that the hyperbolic heat kernel $K_{\mathbb{H}}(t;z,w)$ only depends on the {hyperbolic distance} $d_{\mathrm{hyp}}(z,w)$. If $z=w$, i.e., if $d_{\mathrm{hyp}}(z,w)=0$, we have
\begin{align*}
K_{\mathbb{H}}(t;z,z)= \frac{1}{2\pi}\int\limits\limits_0^\infty e^{-(r^2+ 1\slash 4)t}r \tanh (\pi r) \,dr.
\end{align*}
The hyperbolic heat kernel $K_{\mathrm{hyp}}(t;z,w)$ on $Y$ is a function of $t\in\mathbb{R}_{> 0}$ and $z, w\in Y$ and can be defined by averaging over the group $\Gamma$, namely
\begin{equation*}
 K_{\mathrm{hyp}}(t;z,w)=\sum_{\gamma\in\Gamma}K_{\mathbb{H}}(t;z,\gamma w).
\end{equation*}
For notational convenience we set $$K_{\mathrm{hyp}}(t;z):=K_{\mathrm{hyp}}(t;z,z).$$
The hyperbolic heat kernel $K_{\mathrm{hyp}}(t;z,w)$ satisfies the following heat equation
\begin{align*} 
&\bigg(\del + \frac{\partial}{\partial t}\bigg)K_{\mathrm{hyp}}
(t;z,w) =0\qquad (w\in Y).
\end{align*}
For any $C^\infty$-function $f(w)$ on $Y$, the hyperbolic heat kernel satisfies the relation
\begin{align*}
& \lim_{t \to 0}\int\limits\limits_Y K_{\mathrm{hyp}}(t;z,w) f(w) \hyp(w)= f(z) \qquad (z\in Y).
\end{align*}
Moreover, for $t\in\mathbb{R}_{> 0}$ and $w\in \mathbb{H}$, we have
\begin{align*}
\int\limits_Y K_{\mathrm{hyp}}(t;z,w) \hyp(z)=1.
\end{align*}
The hyperbolic heat kernel admits the following spectral expansion (see, e.g., \cite{Ch}, pp.~108--112)
\begin{align*}
&K_{\mathrm{hyp}}(t;z,w) = \sum_{n=0}^{\infty}e^{-\lambda_{n}t} \varphi_{n}(z)\varphi_{n}(w)+
\frac{1}{2\pi}\sum_{j=1}^{p_\Gamma} \int\limits_{0}^{\infty}e^{-(r^{2}+ 1\slash4)t}{E}_{p_j}(z,1
\slash2+ir){E}_{p_j}(w,1\slash 2-ir)dr
\end{align*}
in terms of an orthonormal basis $\lbrace\varphi_n(z)\rbrace_{n=0}^{\infty}$ of eigenfuntions associated to the discrete spectrum $\lbrace\lambda_n\rbrace_{n=1}^\infty$ of $\del$ and in terms of the eigenfunctions $\lbrace E_{p_j}(z,1\slash 2+ir)\rbrace_{j=1}^{p_\Gamma}$ associated to the continuous spectrum $\{1\slash 4+r^2 \mid r\in\mathbb{R}\}$ of $\del$.

\vspace{.1cm}\noindent
To define the Selberg zeta constant we first recall that, for $s \in \mathbb{C}$ with $\Re(s)>1$, the Selberg zeta function $Z_{Y}(s)$ on $Y$ is defined via the Euler product expansion 
\begin{align*}
Z_{Y}(s)= \prod_{\gamma \in \mathcal{H}_\Gamma}Z_\gamma(s),
\end{align*}
where the local factors $Z_\gamma(s)$ are given by
\begin{align*}
Z_\gamma(s)= \prod_{n=0}^\infty\big( 1- e^{-(s+n)\ell_\gamma}\big)
\end{align*}
with $\ell_\gamma$ denoting the hyperbolic length of the closed geodesic determined by $\gamma \in \mathcal{H}_\Gamma$. 
The Selberg zeta function $Z_{Y}(s)$ admits a meromorphic continuation to all $s\in \mathbb{C}$ with zeros and poles characterized by the spectral theory of the hyperbolic Laplacian (see, e.g., \cite{hej}, Theorem 5.3). For our purpose, it suffices to know that the logarithmic derivative of $Z_{Y}(s)$ has a simple pole at $s=1$. We call the constant
\begin{align}\label{selbergzetaconstant}
c_Y:= \lim_{s\to 1}\bigg(\frac{Z'_{Y}(s)}{Z_{Y}(s)}-\frac{1}{s-1}\bigg)
\end{align}
the Selberg zeta constant of $Y$. 
It can be expressed in terms of the hyperbolic heat trace $\mathrm{H}\mathrm{Tr}K_{\mathrm{hyp}}(t)$, which is a well-defined function for $t\in \mathbb{R}_{\geq 0}$ defined by 
\begin{equation*}
 \mathrm{H}\mathrm{Tr}K_{\mathrm{hyp}}(t):= \int\limits_{Y}\mathrm{H}K_{\mathrm{hyp}}(t;z)\hyp(z),
\end{equation*} 
where
\begin{align}\label{heatkernelhyperbplic}
\mathrm{H}K_{\mathrm{hyp}}(t;z)=\sum_{\substack{\gamma \in \Gamma\\\gamma\, \mathrm{hyperbolic}}}K_{\mathbb{H}}(t;z,\gamma z).
\end{align}
Then, we have the relation (see \cite{JK01})
\begin{align*}
c_Y-1=\int\limits_0^\infty (\mathrm{H}\mathrm{Tr}K_{\mathrm{hyp}}(t)-1)\,dt.
\end{align*}
In \cite{JK01} upper and lower bounds for the constant $c_Y$ are proven, e.g., see Remark \ref{r2}.

\subsection{Hyperbolic and canonical Green's functions}\label{sec_hypandcangreens}

For $z, w\in \mathbb{H}$ with $z\not=w$ and $s\in\mathbb{C}$ with $\mathrm{Re}(s)>0$, we define
\begin{equation*} 
\mathcal{G}_{\mathbb{H},s}(z,w)= \dfrac{\Gamma(s)^{2}}{\Gamma(2s)}u(z,w)^{-s}
F\left(s,s;2s;-u(z,w)^{-1}\right),
\end{equation*} 
where $u(z,w)$ is the point-pair invariant given in \eqref{defn_u}, and $F(s,s;2s;Z)$ is the Gaussian hypergeometric function. Letting $s=1$, we obtain the Green's function $\mathcal{G}_{\mathbb{H}}(z,w):=\mathcal{G}_{\mathbb{H},1}(z,w)$ on $\mathbb{H}$.
Using the well-known formula $F(1,1;2;-Z)=\log(Z+1)/Z$, for $z, w\in \mathbb{H}$ with $z\not=w$, we get 
\begin{align}\label{greenH}
\mathcal{G}_{\mathbb{H}}(z,w)=\log\left(1+u(z,w)^{-1}\right)=
-\log \bigg|\frac{z-{w}}{z-\overline{w}}\bigg|^2.
\end{align}
The Green's function on $\mathbb{H}$ is related to the hyperbolic heat kernel on $\mathbb{H}$ through the following formula 
\begin{align}\label{heatkernelandgreen}
\mathcal{G}_{\mathbb{H}}(z,w)= 4 \pi \int\limits_0^\infty K_{\mathbb{H}}(t;z,w)\, dt,
\end{align}
where $z, w\in \mathbb{H}$ with $z\not=w$.

For $z, w \in  Y$ with $z\not=w$ and $s\in \mathbb{C}$ with $\Re(s)>1$, the automorphic Green's function $\mathcal{G}_{\mathrm{hyp},\, s}(z,w)$ is defined by
\begin{align*}
\mathcal{G}_{\mathrm{hyp},\, \hspace{-.03cm} s}(z,w)= \sum_{\gamma \in \Gamma}\mathcal{G}_{\mathbb{H},\, \hspace{-.03cm} s}(z,\gamma w).
\end{align*}
The automorphic Green's function $\mathcal{G}_{\mathrm{hyp},\, \hspace{-.03cm} s}(z,w)$ is holomorphic for $s\in \mathbb{C}$ with $\Re(s)>1$, and admits a meromorphic continuation to all $s\in \mathbb{C}$ with a simple pole at $s=1$ with the residue $4\pi \slash v_Y$.

For $z, w\in Y$ with $z\not=w$, the (rescaled) hyperbolic Green's function $\mathcal{G}_{\mathrm{hyp}}(z,w)$, associated with the rescaled hyperbolic metric $\shyp$, equals the constant term in the Laurent expansion of $\mathcal{G}_{\mathrm{hyp},\, s}(z,w)$ at $s=1$, i.e.,
\begin{align}\label{laurentexpanionautogreen}
\mathcal{G}_{\mathrm{hyp}}(z,w)= \lim_{s \to 1}\bigg(\mathcal{G}_{\mathrm{hyp},\, \hspace{-.03cm} s}(z,w)- \frac{4\pi}{s(s-1)v_Y}\bigg).
\end{align}
For $z, w\in Y$ with $z\not=w$, the hyperbolic Green's function is related to the hyperbolic heat kernel through the following formula
\begin{align*}
\mathcal{G}_{\mathrm{hyp}}(z,w)= 4\pi \int\limits_0^\infty\bigg( K_{\mathrm{hyp}}(t;z,w)-\frac{1}{v_Y}\bigg) dt.
\end{align*}
One can show that the hyperbolic Green's function $\mathcal{G}_{\mathrm{hyp}}(z,w)$ satisfies the differential equation
\begin{align*}
d_z d_z^c \, \mathcal{G}_{\mathrm{hyp}}(z,w)+ \delta_w(z)=\shyp(z),
\end{align*}
where $d_z=\left(\partial_z + \overline{\partial}_z \right), d_z^{c}= \left( \partial_z - \overline{\partial}_z
\right)\slash 4\pi i$, and $d_zd_z^{c}= -{\partial_z\overline{\partial}_z}\slash{2\pi i}$. The $\delta_{w}(z)$ is the Dirac delta distribution. The hyperbolic Green's function $\mathcal{G}_{\mathrm{hyp}}(z,w)$ also satisfies the following normalization condition 
\begin{align}\label{normalizationhyperbolicgreen}
\int\limits_Y \mathcal{G}_{\mathrm{hyp}}(z,w) \hyp(z)=0.
\end{align}
Let $S_{2}(\Gamma)$ denote the $\mathbb{C}$-vector space of cusp forms of weight $2$ with respect to $\Gamma$ equipped with the Petersson inner product
\begin{align*}
\langle f, g \rangle_\mathrm{pet}:= \int\limits_Y f(z) \overline{g(z)}\Im(z)^2 \hyp(z)\,\ 
\end{align*}
for $f, g\in S_{2}(\Gamma)$.
Let $\left\lbrace f_{1},\ldots,f_{g_Y}\right\rbrace $ denote an orthonormal basis of $S_{2}(\Gamma)$ with respect to the {Petersson 
inner product}. Then, the canonical metric on $Y$  is defined by
\begin{equation*}
 \can(z)=\frac{i}{2g_Y} \sum_{j=1}^{g_Y}\left|f_{j}(z)\right|^{2}dz\wedge d\overline{z}.
\end{equation*}
Note that $\can$ extends smoothly to the canonical metric $\can$ on $X$.
The canonical Green's function $\mathcal{G}_{\mathrm{can}}(z,w)$ on $X$ is a function on $X\times X$, which is smooth away from the diagonal and has a logarithmic singularity along the diagonal. Away from the diagonal, it is uniquely characterized by 
\begin{align*}
d_z d_z^c\, \mathcal{G}_{\mathrm{can}}(z,w)+ \delta_{w}(z)=\can(z), \,\ \text{where} \,\ z, w\in X
\end{align*}
with the normalization condition
\begin{align*}
 \int\limits_{X}\mathcal{G}_{\mathrm{can}}(z,w)\can(z)=0 \,\ \text{with}\,\ w\in X.
\end{align*}
 In his work \cite{A12}, Aryasomayajula proved that the restriction of $\mathcal{G}_{\mathrm{can}}(z,w)$ to $Y\times Y$ coincides with the canonical Green's function on $Y$.
In \cite{Lang88}, p.~26, one can find an  explicit formula for the canonical Green's function on quotient spaces of genus zero having elliptic
fixed points.
\subsection{A key identity}
In this subsection, we recall a closed-form expression for the canonical Green's function in terms of hyperbolic Green's function and analytic functions derived from the hyperbolic heat kernel. Originally these ideas come from Jorgenson--Kramer \cite{JK06}, where the authors proved bounds for the canonical Green's function for compact Riemann surface associated with a cofinite Fuchsian subgroup having neither cusps nor elliptic fixed points. Later, Aryasomayajula \cite{A12}, \cite{A15Z}, extended these bounds for the canonical Green's function to non-compact Riemann surface associated with an arbitrary cofinite Fuchsian subgroup. 


Let $\mathcal{G}_{\mathrm{hyp}}(z,w)$ and $\mathcal{G}_{\mathrm{can}}(z,w)$ be the hyperbolic and the canonical Green's functions recalled in subsection \ref{sec_hypandcangreens}. Then, for $z, w\in X\setminus \mathcal{P}_\Gamma$, we have (see \cite{A12}, Proposition 2.6.4.)
\begin{align*}
&\mathcal{G}_{\mathrm{hyp}}(z,w)-\mathcal{G}_{\mathrm{can}}(z,w)= \Phi(z)+\Phi(w),
\end{align*}
where (from Corollary 3.2.7 in \cite{A12}), the function $\Phi(z)$ is given by the
formula
\begin{align*}
\Phi(z)=\frac{1}{2g_Y}\int\limits_{Y}\mathcal{G}_{\mathrm{hyp}}(z, \zeta)\, F(\zeta)\hyp(\zeta)-\frac{\mathrm{C}_{\mathrm{hyp}}}{8g^2_\Gamma}
\end{align*}
with
\begin{align*}
F(\zeta):=\int\limits_{0}^{\infty}\Delta_{\mathrm{hyp},\, \hspace{-.03cm} \zeta}
K_{\mathrm{hyp}}(t;\zeta)dt
\end{align*}
and
\begin{align}\label{Chypdefn}
\mathrm{C}_{\mathrm{hyp}}:=\int\limits_{Y\times Y} \mathcal{G}_{\mathrm{hyp}}(\xi, \zeta)\, F(\xi)\, F(\zeta)\hyp(\xi)\hyp(\zeta).
\end{align}

We set
\begin{align}
&\mathrm{E}_\Gamma(z)= \sum_{\substack{\gamma \in \Gamma\backslash \{\id\} \\\gamma\, \mathrm{elliptic}}} \mathcal{G}_{\mathbb{H}}(z, \gamma z), \label{ellipticseries} \\
&\mathrm{P}_\Gamma(z)= \sum_{\substack{\gamma \in \Gamma\backslash \{\id\} \\\gamma\, \mathrm{parabolic}}} \mathcal{G}_{\mathbb{H}}(z, \gamma z),\notag \\&\mathrm{H}_\Gamma(z)= 4\pi\int\limits_{0}^{\infty}\bigg(\mathrm{H}K_{\mathrm{hyp}}(t;z)-\frac{1}{v_Y}\bigg)dt\notag,
\end{align}
where the Green's function $\mathcal{G}_{\mathbb{H}}(z, w)$ is given by \eqref{heatkernelandgreen} and the function $\mathrm{H}K_{\mathrm{hyp}}(t;z)$ is defined in \eqref{heatkernelhyperbplic}. In \cite{JK11}, the authors proved that the functions $\mathrm{E}_\Gamma(z)$, $\mathrm{P}_\Gamma(z)$, and $\mathrm{H}_\Gamma(z)$ are absolutely and locally uniformly convergent.

\begin{prop}\label{awayellipticprop}
With the above notation, for $z, w \in X \setminus (\mathcal{E}_\Gamma \cup \mathcal{P}_\Gamma)$, $z\not = w$,
we have
\begin{align*}
&\mathcal{G}_{\mathrm{hyp}}(z,w)-\mathcal{G}_{\mathrm{can}}(z,w)= \Phi(z)+\Phi(w),
\end{align*}
where
\begin{align*}
\Phi(z)
&=\frac{1}{8\pi g_Y}\int\limits_{Y}\mathcal{G}_{\mathrm{hyp}}(z,\zeta)\left(\Delta_{\mathrm{hyp}, \, \hspace{-.03cm} \zeta}\mathrm{P}_\Gamma(\zeta)\right)
\hyp(\zeta) 
-\frac{1}{2g_Y} \sum_{j=1}^{e_\Gamma}\left(1- \frac{1}{\ord(e_j)}\right) \mathcal{G}_{\mathrm{hyp}}(z, e_j)
\\
&\phantom{=}+\frac{\mathrm{H}_\Gamma(z)+\mathrm{E}_\Gamma(z)}{2g_Y}- \frac{\mathrm{C_{hyp}}}{8g_Y^{2}}-\frac{2\pi(c_{Y}-1)}{g_Y v_Y}
-\frac{1}{2g_Y}\int\limits_{Y}\mathrm{E}_\Gamma(\zeta)\shyp(\zeta).
\end{align*}
Here $\mathrm{E}_\Gamma(z)$, $\mathrm{P}_\Gamma(z)$, $\mathrm{H}_\Gamma(z)$ are defined in \eqref{ellipticseries}, $\mathrm{C_{hyp}}$ is given by \eqref{Chypdefn}, and $c_Y$ denotes the Selberg zeta constant of $Y$ given in \eqref{selbergzetaconstant}.
\end{prop}
\begin{proof}
See \cite{A15Z}, Corollary 3.12.
\end{proof}

\subsection{Congruence subgroup \texorpdfstring{$\Gamma_0(N)$}{}}
Let $N$ be a positive integer. We define
\begin{align*}
\Gamma_0(N)&=\bigg \lbrace\bigg( \begin{array}{ccc} a &b\\
 c & d \end{array}\bigg)\in \PSLZ\mid c \equiv 0 \,\ \text{mod}\,\ N\bigg \rbrace.
 \end{align*}
 In this subsection we consider $\Gamma= \Gamma_0(N)$, and $Y=\Gamma_0(N)\backslash \mathbb{H}$. Then the hyperbolic volume of $Y$ is given by
\begin{align}\label{volumeGammaN}
 v_Y= \frac{\pi N}{3}\prod_{\substack{p|N\\ p\,\text{prime}}}\left(1+\frac{1}{p}\right).
\end{align}
We know that every cusp of $\Gamma$ is equivalent to one among the following rationals 
\begin{align*}
\frac{m}{n}\,\ \text{with}\,\ m,n>0,\,\ n|N, \,\ (m, n)=1.
\end{align*}
Two cusps $m\slash n$ and $m_1\slash n_1$ of the above type are $\Gamma$-equivalent if and only if 
\begin{align*}
n_1=n\,\ \text{and}\,\ m_1\equiv m\mod\left(n, \frac{N}{n}\right).
\end{align*}
Hence the number of inequivalent cusps of $\Gamma$ is given by
\begin{align}\label{numberofcusp}
p_{\small{\Gamma}}= \sum_{\substack{d|N\\d>0}}\phi\left(\left(d, {N}\slash{d}\right)\right),
\end{align}
where $\phi$ is the Euler function.

\begin{rem}
    A cusp of $Y$ is the $\Gamma$-orbit of a parabolic fixed point of $\Gamma$.
By $\mathcal{P}_\Gamma \subseteq \mathbb{P}^1(\mathbbm{Q})$ we denote a complete set of representatives
for the cusps of $Y$. We will always identify a cusp of $Y$ with its representative in $\mathcal{P}_{\Gamma}$.
Hereby, identifying $\mathbb{P}^1(\mathbbm{Q})$ with $\mathbb{Q}\cup\{\infty\}$,
we write elements of $\mathbb{P}^1(\mathbbm{Q})$ as $n/m$ for $n,m\in\mathbb{Z}$, not both
equal to $0$, and we always assume that $n|N$ and $(n,m)=1$; we set $1/N:=\infty$.
\end{rem}

\vspace{0.1cm} \noindent
Let $a=m/n$ be a cusp of $\Gamma$, the scattering function $\varphi_{a\infty}(s)$ with respect to the cusps $a, \infty$, is given by the following formula (see \cite{DI82}, p.~247)
\begin{align}\label{scatteringfunctionGamma0Nainfty}
\varphi_{a\infty}(s)=\sqrt{\pi}\frac{\Gamma(s-\frac{1}{2})}{\Gamma(s)}\frac{\zeta(2s-1)}{\zeta(2s)}\frac{\,\phi(n)}{\phi((n, N\slash n))}F(s),
\end{align}
where
\begin{align*}
F(s) =\left(\frac{(n, N\slash n)}{nN}\right)^s\prod_{\substack{p|N\\ p \, \text{prime}}}\frac{p^{2s}}{p^{2s}-1}\prod_{\substack{q|\frac{N}{n}\\q\, \text{prime}}}\left(1-\frac{1}{q^{2s-1}}\right).
\end{align*}
By following the line of proof from \cite{DI82} we get the scattering function $ \varphi_{a0}(s)$
 with respect to the cusps $a=m\slash n,\, 0$, and which is given by the following formula (see \cite{PM}, Chapter 3, pp.~48--49)
\begin{align}\label{phia0Gammagen0N}
\varphi_{a0}(s)=\sqrt{\pi}\frac{\Gamma(s-\frac{1}{2})}{\Gamma(s)}\frac{\zeta(2s-1)}{\zeta(2s)}\frac{\phi(N\slash n)}{\phi(n, N\slash n)}\,G(s),
\end{align}
where
\begin{align*}
G(s)=\frac{(n^2,N)^s}{N^{2s}}\prod_{\substack{p|N\\ p\,\text{prime}}}\frac{p^{2s}}{p^{2s}-1}\prod_{\substack{q|n\\ q\,\text{prime} }}\left(1-\frac{1}{q^{2s-1}}\right).
\end{align*}
The elliptic fixed points of $\Gamma$ of order $\ord(e_j)=2$ are explicitly given by (see, e.g., \cite{DS})
\begin{align}\label{nfori}
e_j= \frac{n+i}{n^2+1}\,\ \text{for}\,\ n=0,\dotsc, N-1 \,\ \text{such that}\,\ n^2+1 \equiv 0 \,(\text{mod}\, N)
\end{align}
and the elliptic fixed points of $\Gamma$  with $\ord(e_j)=3$ are explicitly given by
\begin{align}\label{nforother}
e_j= \frac{n+\frac{1+i\sqrt{3}}{2}}{n^2-n+1}\,\ \text{for}\,\ n=0,\dotsc, N-1 \,\ \text{such that}\,\ n^2-n+1 \equiv 0 \,(\text{mod}\, N).
\end{align}

The number of elliptic fixed points is equal to 
\begin{align*}
e_{\small{\Gamma}}= \nu_2+\nu_3,
\end{align*}
where
\begin{equation*}
    \nu_2=
    \begin{cases}
     0 & \text{if}\,\ 4|N, \\
     \displaystyle \prod_{\substack{p|N\\ p\,\text{prime}}}\small{\left(1+\left(\frac{-1}{p}\right)\right) }& \text{otherwise},
    \end{cases}
\end{equation*}
and
\begin{equation*}
    \nu_3=
    \begin{cases}
     0 & \text{if}\,\ 9|N, \\
     \displaystyle \prod_{\substack{p|N\\ p\,\text{prime}}}\small{\left(1+\left(\frac{-3}{p}\right)\right)}& \text{otherwise.}\,\ 
    \end{cases}
\end{equation*}
Here $\nu_2$ resp.~$\nu_3$ denotes the number of elliptic points of order $2$ and $3$, respectively, and $\left(\frac{\cdot}{p}\right)$ denotes the generalized quadratic residue symbol (see \cite{Sh}, p.~25).
\begin{rem}\label{r1}
Note that $\omega(N)=O ({\log N }\slash{\log \log N})$, where $\omega(N)$ denotes the number of primes which divide $N$.
This implies (see, e.g., \cite{A15Z}, Remark 6.12)
\begin{align}\label{ellipticboundGammaN}
e_{\small{\Gamma}}=O\big(N^\varepsilon\big)\,\ \text{for any given}\,\ \varepsilon>0.
\end{align}
Then, from the volume formula (\ref{volumegenusformula}), we have the following asymptotics
\begin{align}\label{volgeGamma}
\frac{4\pi (g_Y-1)}{v_Y}= 1+o(N)\,\ \text{as}\,\ N\to \infty.
\end{align}

\end{rem}

\begin{rem}\label{r2} We also recall two important bounds for $Y$ by Jorgenson--Kramer. In \cite{JK01}, pp.~26--27, they proved the following bound for the Selberg zeta constant 
    \begin{align}\label{cX}
        c_{Y}= O_\epsilon(N^\epsilon)\,\ \text{for any given}\,\ \varepsilon>0.
        \end{align}
        In  \cite{JK09}, Proposition 5.4, the authors prove a bound for the term $\displaystyle d_{Y}=\sup_{z\in Y}\left|\frac{\can(z)}{\shyp(z)}\right|$, which is given by
        \begin{align}\label{dX}
            d_Y=O(1),
        \end{align}
        where the implied constant is
independent of $N$. 
        \end{rem}
        \begin{rem}\label{lamda}
           Let $\lambda_1$ denote the smallest non-zero eigenvalue of the hyperbolic Laplacian $\del$. From \cite{LRS95}, Theorem 1.1, we know $
\lambda_1\geq 21\slash 100$.        
\end{rem}

We use the bound mentioned in Remark \ref{r1}, Remark \ref{r2} and Remark \ref{lamda} later in our paper, namely, in section \ref{BSCGN} and in section \ref{BKLFGN} and they play a crucial role to prove Theorem \ref{maintheorem2article}.

\section{Proof of Theorem \ref{maintheorem1article}}\label{section5}

To prove the Main Theorem \ref{maintheorem1article} we start with the key identity
given in Proposition \ref{awayellipticprop} that expresses the canonical Green's function in terms of the hyperbolic Green's function and explicit analytic functions.
We then first recall the behaviour hyperbolic Green's function and the Kronecker limit function at a cusps,
which can be proven by using its Fourier expansions.
We then recall and provide bounds for the other terms arising in Proposition \ref{awayellipticprop}. 
Employing all these results, we are finally able to deduce Theorem \ref{maintheorem1article}.



\begin{lem}\label{ghypsum}
Let $p_j\in \mathcal{P}_\Gamma$ be a cusps with scaling matrix $\sigma_{\small{p_j}}$. For fixed $w\in Y$, as $z$ approaches to the cusp $p_j\in \mathcal{P}_\Gamma$, we have
\begin{align*}
&\mathcal{G}_{\mathrm{hyp}}(z, w)=4\pi\mathcal{K}_{p_j}(w) - 
\frac{4\pi}{v_Y}-\frac{4\pi \log\big(\hspace{-0.03cm}\Im({\sigma_{p_j}^{-1}}z)\big)}{v_Y},
\end{align*}
where $\mathcal{K}_{p_j}(w)$ denotes the Kronecker limit function given in \eqref{KLFdefn}.
\end{lem}

\begin{proof}
From \cite{A12}, Corollary 1.9.5, we know the following estimate for the automorphic Green's function. Let $p_j$ and $p_k$ are two cusps, then
\begin{align*}
&\mathcal{G}_{\mathrm{hyp},\, \hspace{-.03cm} s}(\sigma_{p_j}^{}z,\sigma_{p_k}^{}w)\\&=
\frac{4\pi\Im(z)^{1-s}}{2s-1}{E}_{p_j}(\sigma_{p_k}^{}w,s)
 - \delta_{p_jp_k}\log\big{|}1 - e^{2\pi i(z- w)}
\big{|}
+ O\big(e^{-2\pi(\Im(z)-\Im(w))}\big).
\end{align*}
This estimate implies 
\begin{align*}
&\mathcal{G}_{\mathrm{hyp},\, \hspace{-.03cm} s}(z,w)\\
&=\frac{4\pi\Im(\tiny{\sigma_{p_j}^{-1}}z)^{1-s}}{2s-1}{E}_{p_j}(w,s)
 - \log\big{|}1 - e^{2\pi i(\sigma_{p_j}^{-1}z- \sigma_{p_j}^{-1}w)}
\big{|}+ O\big(e^{\hspace{-.01cm}-2\pi(\Im(\sigma_{p_j}^{-1}z)-\Im(\sigma_{p_j}^{-1}w))}\big).
\end{align*}
Now using (\ref{KLFdefn}) and (\ref{laurentexpanionautogreen}), we can write 
\begin{align*}
&\mathcal{G}_{\mathrm{hyp}}(z, w)=4\pi\mathcal{K}_{p_j}(w) - 
\frac{4\pi}{v_Y}-\frac{4\pi \log\big(\hspace{-0.03cm}\Im({\sigma_{p_j}^{-1}}z)\big)}{v_Y}\,\ \text{as}\,\ z \to p_j.
\end{align*}
This completes the proof.
\end{proof}

\begin{lem}\label{FourierexpansionKLF} 
Let $p_k, p_l \in \mathcal{P}_\Gamma$ $($with $p_k\not=p_l)$ be cusps with scaling matrices $\sigma_{\small{p_k}},\sigma_{\small{p_l}}$, respectively. Then, 
 that as $z$ approaches the cusp $p_l$, we have
\begin{align*}
\mathcal{K}_{p_k}(z)=\mathcal{C}_{p_k p_l}-\frac{\log\big(\hspace{-0.03cm}\Im(\sigma_{p_l}^{-1}z)\big)}{v_Y},
\end{align*}
where $\mathcal{C}_{p_k p_l}$ is the scattering constant defined in \eqref{scatteringconstantdef}.
\end{lem}
\begin{proof}
By definition \eqref{KLFdefn}, we recall the identity
\begin{align*}
\mathcal{K}_{p_k}(\sigma_{\small{p_l}}^{}z)= \lim_{s \to 1}\bigg({E}_{p_k }(\sigma_{\small{p_l}}^{}z)- \frac{1}{(s-1)v_Y}\bigg).
\end{align*}
Now, using the Fourier expansion \eqref{FourierEisenstein} of the Eisenstein series, we obtain 
\begin{align}\label{limitforn0}
\mathcal{K}_{p_k}(\sigma_{\small{p_l}}^{}z)=& \lim_{s \to 1}\bigg(\varphi_{p_kp_l}(s)  \Im(z)^{1-s}- \frac{1}{(s-1)v_Y}\bigg)\notag\\
&+
\sum_{n\not=0}\varphi_{p_kp_l}(n,1)\Im(z)^{1/2}K_{1/2}(2\pi|n|\Im(z))\,e^{2\pi i n\Re(z)}.
\end{align}
Recalling \eqref{scatteringconstantdef}, at $s=1$, we have the Laurent expansion
\begin{align*}
\varphi_{p_kp_l}(s)=\frac{1}{(s-1)v_Y}+ \mathcal{C}_{p_k p_l}+ O(s-1).
\end{align*}
Further, at $s=1$, we have the Taylor expansion
\begin{align*}
\Im(z)^{1-s}=1-\log\big(\Im(z)\big) (s-1)+O\big((s-1)^2\big).
\end{align*}
This yields
\begin{align}\label{eq_klf_fourier1}
\lim_{s \to 1}\bigg(\varphi_{p_kp_l}(s)  \Im(z)^{1-s}- \frac{1}{(s-1)v_Y}\bigg)=\mathcal{C}_{p_k p_l}-\frac{\log \big(\Im(z)\big)}{v_Y}.
\end{align}
Moreover, using the well-known identity $K_{1\slash 2}(Z)=\pi^{1/2}(2Z)^{-1/2}e^{-Z}$ with $Z:=2\pi |n| \Im(z)$, we get
\begin{align}\label{eq_klf_fourier2}
\Im(z)^{1/2}
K_{1\slash 2}(2\pi |n| \Im(z) )=\frac{1}{2}|n|^{-1/2}e^{-2\pi |n|\Im(z)}.
\end{align}
Substituting \eqref{eq_klf_fourier1} and \eqref{eq_klf_fourier2} into \eqref{limitforn0} 
yields that the Fourier expansion of the Kronecker limit function $\mathcal{K}_{p_k}(z)$ with respect to the cusp $p_l$ has the form
\begin{align*}
\mathcal{K}_{p_k}(\sigma_{\small{p_l}} z)=\mathcal{C}_{p_k p_l}-\frac{\log\big(\Im(z)\big)}{v_Y}+\frac{1}{2}\sum_{n\not=0}
\varphi_{p_kp_l}(n,1)|n|^{-1/2} \,e^{-2\pi |n| \Im(z)} \,  e^{2\pi i n \Re(z)},
\end{align*}
where $\varphi_{p_kp_l}(n,1)$ is given by \eqref{phin}. This yields the assertion.
\end{proof}

\begin{prop}\label{propparabolicintegral}
Let $p_k$ be a cusp of $\Gamma$. Then, as $z$ approaches the cusp $p_k$, we have
\begin{align*}
&\frac{1}{8\pi g_Y}\int\limits_{Y}\mathcal{G}_{\mathrm{hyp}}(z,w)\big(\Delta_{\mathrm{hyp},\, \hspace{-.03cm} w} \mathrm{P}_\Gamma(w)\big)
\hyp(w)\\
&=\frac{2\pi p_\Gamma- v_Y}{g_Y v_Y}\log\big(\Im(\sigma_{p_k}^{-1}z)\big)-\frac{2\pi}{g_Y}\sum_{j=1}^{p_\Gamma}\mathcal{C}_{p_kp_j}-\frac{\log(4\pi)}{g_Y}+\frac{2\pi p_\Gamma}{g_Y v_Y}+R,
\end{align*}
where $|R|\leq \big(\frac{4\pi^2}{3}+1\big)\frac{p_\Gamma}{g_Y v_Y}$.
\begin{proof}
The complete proof of this proposition is in \cite{PM}, Proposition 2.4.7 but for the reader's convenience we provide a very short outline of the proof here. Note that, in \cite{A12}, Chapter 7, the author considered the Fuchsian subgroup $\Gamma$ without any elliptic fixed points. To prove \cite{A12}, Proposition 7.1.12, the author used the following identity
\begin{align*}
    \frac{v_Y}{2\pi}-p_\Gamma= 2(g_Y-1)
\end{align*}
solely in one instance on page 149. By employing this identity once more in \cite{A12}, Proposition 7.1.12, we obtain 
\begin{align}\label{formula}
    &\frac{1}{8\pi g_Y}\int\limits_{Y}\mathcal{G}_{\mathrm{hyp}}(z,w)\big(\Delta_{\mathrm{hyp},\, \hspace{-.03cm} w} \mathrm{P}_\Gamma(w)\big)
\hyp(w)\notag \\
&=\frac{2\pi p_\Gamma- v_Y}{g_Y v_Y}\log\big(\Im(\sigma_{p_k}^{-1}z)\big)-\frac{2\pi}{g_Y}\sum_{j=1}^{p_\Gamma}\mathcal{C}_{p_kp_j}-\frac{\log(4\pi)}{g_Y}+\frac{2\pi p_\Gamma}{g_Y v_Y}+R,
\end{align}
where
\begin{align*}
    R= \frac{1}{2g_Y}\sum_{j=1}^{p_\Gamma}\int_{1\slash \Im(\sigma_{p_k}^{-1}z)}^{\infty}\frac{\log (\Im(w))}{v_Y}\Delta_{\mathrm{hyp},\, \hspace{-.03cm} w} \mathrm{P}_{\mathrm{gen},\, p_j}(\sigma_{p_j}w)\frac{d(\Im(w))}{\Im(w)^2}.
    \end{align*}
Note that, here $\mathrm{P}_{\mathrm{gen},\, p_j}(w):=\displaystyle\sum_{n\not=0}\mathcal{G}_\mathbb{H}(w, \gamma_{p_j}^n w)$, where $\gamma_{p_j}$ is a generator of the stabilizer subgroup $\Gamma_{p_j}$. Then from \cite{A12}, Lemma 7.1.9, we have 
\begin{align*}
    |R|\leq \left(\frac{4\pi^2}{3}+1\right)\frac{p_\Gamma}{g_Y v_Y}.
    \end{align*}
    By employing the above bound in the formula \eqref{formula} we complete the proof.
\end{proof}
\end{prop}

\begin{lem}\label{ellipticandhyperbolic}
Let $p_k\in \mathcal{P}_\Gamma$ be a cusp with scaling matrix $\sigma_{\small{p_k}}$.
Then, as $z$ approaches the cusp $p_k$, we have
\begin{align*}
\frac{\mathrm{H}_\Gamma(z)+\mathrm{E}_\Gamma(z)}{2g_Y}=\frac{2\pi}{g_Y}\, \mathcal{C}_{p_k p_k} -\frac{4\pi\log\big(\Im(\sigma_{p_k}^{-1}z)\big)}{g_Y v_Y}-\frac{2\pi}{g_Y v_Y},
\end{align*}
where $\mathrm{E}_\Gamma(z)$ and $\mathrm{H}_\Gamma(z)$ are given in \eqref{ellipticseries}.
\begin{proof}
This result is proven in \cite{A15Z}, Proposition 2.10.,
with a minus sign corrected.
\end{proof}
\end{lem}
\begin{lem}\label{Chyp}
The constant $\mathrm{C_{hyp}}$ given in (\ref{Chypdefn}), satisfies the inequality 
\begin{align*}
\frac{\mathrm{C_{hyp}}}{8g_Y^{2}}
\leq  \frac{2\pi \left(d_{Y}+1\right)^{2}}{\lambda_{1}v_Y},
\end{align*}
where
$$
d_{Y}:=\sup_{z\in Y}\bigg|\frac{\can(z)}{\shyp(z)}\bigg|
$$
and $\lambda_1$ denotes the smallest non-zero eigenvalue of $\del$.
\begin{proof}
See \cite{A15Z}, Proposition 5.10.
\end{proof}
\end{lem}

\begin{lem}\label{ellipticlemma}
Let $\mathcal{E}_\Gamma=\lbrace e_j \mid j=1, \dotsc ,e_\Gamma \rbrace$ be the set of  elliptic fixed points of $\Gamma$. Then
\begin{align*}
\bigg|\int\limits_Y \mathrm{E}_\Gamma(\zeta)\shyp(\zeta)\bigg|\leq \frac{4\pi \log 2}{v_Y}\sum_{j=1}^{e_\Gamma}(\ord(e_j)-1),
\end{align*}
where $\mathrm{E}_\Gamma(\zeta)$ is given in (\ref{ellipticseries}).
\begin{proof}
Let $e_j$ be an elliptic fixed point. Then we know that the stabilizer group $\Gamma_{e_j}$ is cyclic, and there exists a scaling matrix
$\sigma_{e_j}\in \PSL$ such that 
\begin{align*}
\Gamma_{e_j}= \langle \gamma_{e_j}\rangle,\,\ \text{where}\,\ \gamma_{e_j}= \sigma_{e_j} \gamma_{i,\,{e_j}}\sigma_{e_j}^{-1}
\end{align*}
with  
\begin{align*}
  \gamma_{i,\,{e_j}}= 
  \begin{pmatrix}\cos (\pi\slash \ord(e_j))& \sin  (\pi\slash \ord(e_j))\\ \\
 -\sin  (\pi\slash \ord(e_j)) & \cos  (\pi\slash \ord(e_j))  \end{pmatrix}.
\end{align*} 
Then we have the following disjoint union decomposition 
\begin{align}\label{decomposition}
&\big\lbrace \gamma \in \Gamma\backslash \{\id\}\mid \gamma \, \mathrm{elliptic}\big\rbrace =  \bigcup_{j=1}^{e_\Gamma} \bigcup_{\eta \in \Gamma_{e_j} \backslash \Gamma} \big(\eta^{-1}\Gamma_{e_j} \eta \backslash \{\id\}\big)\notag\\
&=\bigcup_{j=1}^{e_\Gamma} \bigcup_{\eta \in \Gamma_{e_j} \backslash \Gamma} \bigcup_{\substack{n=1}}^{\ord(e_j)-1} \lbrace\eta^{-1}\gamma_{e_j}^n \eta \rbrace=\bigcup_{j=1}^{e_\Gamma} \bigcup_{\eta \in \Gamma_{e_j} \backslash \Gamma} \bigcup_{\substack{n=1}}^{\ord(e_j)-1} \lbrace \eta^{-1}\sigma_{e_j}\gamma_{i,\,{e_j}}^n \sigma_{e_j}^{-1} \eta \rbrace.
\end{align}
Now using the disjoint union decomposition (\ref{decomposition}), we get
\begin{align}
\mathrm{E}_\Gamma(\zeta)&= \sum_{\substack{\gamma \in \Gamma\backslash \{\id\} \\\gamma\, \mathrm{elliptic}}} \mathcal{G}_{\mathbb{H}}(\zeta, \gamma \zeta)
=\sum_{j=1}^{e_\Gamma}\sum_{\eta \in \Gamma_{e_j} \backslash \Gamma}\sum_{n=1}^{\ord(e_j)-1} \mathcal{G}_{\mathbb{H}}(\zeta, \eta^{-1}\sigma_{e_j}\gamma_{i,\,{e_j}}^n \sigma_{e_j}^{-1} \eta\, \zeta)\notag \\
&=\sum_{j=1}^{e_\Gamma}\sum_{\eta \in \Gamma_{e_j} \backslash \Gamma}\sum_{n=1}^{\ord(e_j)-1} \mathcal{G}_{\mathbb{H}}(\sigma_{e_j}^{-1} \eta \,\zeta, \gamma_{i,\,{e_j}}^n \sigma_{e_j}^{-1} \eta \,\zeta)\label{E(zeta)}.
\end{align}

\vspace{0.1cm} \noindent
By taking integral on (\ref{E(zeta)}), we get
\begin{align*}
\int\limits_Y \mathrm{E}_\Gamma(\zeta) \shyp(\zeta)=  \sum_{j=1}^{e_\Gamma}\sum_{\eta \in \Gamma_{e_j} \backslash \Gamma}\sum_{n=1}^{\ord(e_j)-1} \int\limits_{\mathcal{F}_\Gamma} \mathcal{G}_{\mathbb{H}}(\sigma_{e_j}^{-1} \eta\, \zeta, \gamma_{i,\,{e_j}}^n \sigma_{e_j}^{-1} \eta \,\zeta) \shyp(\zeta).
\end{align*}
Now, note that from \eqref{cosh-hypdist}, for $z, w\in \mathbb{H}$ with $z\not=w$, we have
\begin{align*}
u(z,w)^{-1}
=\frac{2}{\cosh(d_{\mathrm{hyp}}(z, w))-1}=
\sinh\left(\frac{d_{\mathrm{hyp}}(z, w)}{2}\right)^{-2}.
\end{align*}
Thus using \eqref{greenH}, we get for $n=1,\ldots,\ord(e_j)-1$, the following identity
\begin{align}\label{logtanh}
\mathcal{G}_{\mathbb{H}}(\sigma_{e_j}^{-1} \eta\,\zeta, \gamma_{i,\,{e_j}}^n \sigma_{e_j}^{-1} \eta\,\zeta)
&=\log\left(1+u( \sigma_{e_j}^{-1} \eta\,\zeta, \gamma_{i,\,{e_j}}^n\sigma_{e_j}^{-1} \eta\,\zeta)^{-1} \right) \notag \\& 
=- \log \left( \tanh^2 \big({d_{\mathrm{hyp}}(\sigma_{e_j}^{-1} \eta\,\zeta, \gamma_{i,\,{e_j}}^n 
\sigma_{e_j}^{-1} \eta\,\zeta)}\slash {2}\big)\right).
\end{align}
By recalling the triangular inequality of the hyperbolic distance function, we have
\begin{align}\label{dist}
\frac{1}{2} d_{\mathrm{hyp}}(\sigma_{e_j}^{-1} \eta\,\zeta, \gamma_{i,\,{e_j}}^n\sigma_{e_j}^{-1} \eta\,\zeta)&\leq \frac{1}{2}\left( d_{\mathrm{hyp}}( \sigma_{e_j}^{-1} \eta\,\zeta, i)+ d_{\mathrm{hyp}}(i, \gamma_{i,\,{e_j}}^n\sigma_{e_j}^{-1} \eta\,\zeta)\right)\notag\\& =  d_{\mathrm{hyp}}(\sigma_{e_j}^{-1} \eta\,\zeta, i).
\end{align}
Then, using (\ref{logtanh}) and (\ref{dist}), we get
\begin{align}\label{boundrho}
\big|\mathcal{G}_{\mathbb{H}}(\sigma_{e_j}^{-1} \eta\,\zeta, \gamma_{i,\,{e_j}}^n\sigma_{e_j}^{-1} \eta\,\zeta)\big|\leq \big|\log \big( \tanh^2(\rho(\sigma_{e_j}^{-1} \eta\,\zeta) )\big)\big|, \,\ \text{where}\,\ \rho( \sigma_{e_j}^{-1} \eta\,\zeta) := d_{\mathrm{hyp}}(\sigma_{e_j}^{-1} \eta\,\zeta, i) .
\end{align}
Now, using (\ref{boundrho}), we get 
\begin{align*}
\bigg|\int\limits_Y \mathrm{E}_\Gamma(\zeta)\shyp(\zeta)\bigg|
&\leq\sum_{j=1}^{e_\Gamma}\sum_{\eta \in \Gamma_{e_j} \backslash \Gamma}
\sum_{n=1}^{\ord(e_j)-1} \int\limits_{\mathcal{F}_\Gamma} \big|\log \big( \tanh^2(\rho( \sigma_{e_j}^{-1} \eta\,\zeta)) \big)\big| \shyp(\zeta)\\
&=\frac{1}{v_Y}\sum_{j=1}^{e_\Gamma}(\ord(e_j)-1)\sum_{\eta \in \Gamma_{e_j} \backslash \Gamma}\,\ 
\int\limits_{\mathcal{F}_\Gamma} \big|\log \big( \tanh^2(\rho(\sigma_{e_j}^{-1} \eta\,\zeta)) \big)\big| \hyp(\zeta).
\end{align*}
Now we write the hyperbolic metric $\hyp(\zeta)$ in hyperbolic polar coordinates centered at $i$, i.e.,
\begin{align*}
\hyp(\zeta)= \sinh(\rho(\zeta)) \, d\rho \, d \theta,\,\ \text{where}\,\ \rho=\rho(\zeta)\in [0, \infty) \,\ \text{and}\,\ \theta \in [0,2 \pi).
\end{align*}
Then making the substitution $\zeta \mapsto \eta^{-1}\sigma_{e_j} \zeta$, and using the $\PSL$-invariance  of $\hyp(\zeta)$, we get
\begin{align*}
&\sum_{\eta \in \Gamma_{e_j} \backslash \Gamma}\,\ 
\int\limits_{ \mathcal{F}_\Gamma} \big|\log \big( \tanh^2(\rho(\sigma_{e_j}^{-1} \eta\,\zeta)) \big)\big| \hyp(\zeta)
=\int\limits_0^\infty\int\limits_0^{2\pi}|\log(\tanh^2 \rho)| \sinh(\rho)  d \theta d\rho\\
&=2\pi \int\limits_0^\infty|\log(\tanh^2(\rho))| \sinh(\rho) d\rho
=2\pi \int\limits_1^\infty|\log(1-t^{-2})| dt
=4\pi \log(2). 
\end{align*}
So, finally we get 
\begin{align*}
&\bigg|\int\limits_Y \mathrm{E}_\Gamma(\zeta)\shyp(\zeta)\bigg|\\
&\leq \frac{\sum_{j=1}^{e_\Gamma}(\ord(e_j)-1)}{v_Y}\int\limits_0^\infty\int\limits_0^\pi|\log(\tanh^2 \rho)| 2\sinh(\rho)  d \theta d\rho=\frac{\pi\log(16)}{v_Y}\sum_{j=1}^{e_\Gamma}(\ord(e_j)-1).
\end{align*}
 This completes the proof.
\end{proof}
\end{lem}


\begin{proof}[Proof of Theorem \ref{maintheorem1article}]
From Proposition \ref{awayellipticprop}, 
for $z, w \in X \setminus (\mathcal{E}_\Gamma \cup \mathcal{P}_\Gamma)$, $z\not = w$,
we have
\begin{align*}
\mathcal{G}_{\mathrm{can}}(z,w)=\mathcal{G}_{\mathrm{hyp}}(z,w)-\Phi(z)-\Phi(w),
\end{align*}
where
\begin{align}\label{phiz}
\Phi(z)
&=\frac{1}{8\pi g_Y}\int\limits_{Y}\mathcal{G}_{\mathrm{hyp}}(z,\zeta)\left(\Delta_{\mathrm{hyp}, \, \hspace{-.03cm} \zeta}\mathrm{P}_\Gamma(\zeta)\right)
\hyp(\zeta) 
-\frac{1}{2g_Y} \sum_{j=1}^{e_\Gamma}\left(1- \frac{1}{\ord(e_j)}\right) \mathcal{G}_{\mathrm{hyp}}(z, e_j)
\\
&\phantom{=}+\frac{\mathrm{H}_\Gamma(z)+\mathrm{E}_\Gamma(z)}{2g_Y}- \frac{\mathrm{C_{hyp}}}{8g_Y^{2}}-\frac{2\pi(c_{Y}-1)}{g_Y v_Y}
-\frac{1}{2g_Y}\int\limits_{Y}\mathrm{E}_\Gamma(\zeta)\shyp(\zeta).
\end{align}
Now, for a fixed $w \in X \setminus (\mathcal{E}_\Gamma \cup \mathcal{P}_\Gamma)$, we set
\begin{align*}
A(w)=\lim_{z\to p_k}\big( \mathcal{G}_{\mathrm{hyp}}(z,w)-\Phi(z)\big).
\end{align*}
Then 
\begin{align*}
A&(w)=\lim_{z\to p_k}\big(\mathcal{G}_{\mathrm{hyp}}(z,w)-\frac{1}{8\pi g_Y}\int\limits_{Y}\mathcal{G}_{\mathrm{hyp}}(z,\zeta)\big(\Delta_{\mathrm{hyp}, \, \hspace{-.03cm} \zeta} \mathrm{P}_\Gamma(\zeta)\big)\hyp(\zeta) -\frac{\mathrm{H}_\Gamma(z)}{2g_Y}-\frac{\mathrm{E}_\Gamma(z)}{2g_Y}\notag\\
&+\frac{1}{2g_Y} \sum_{j=1}^{e_\Gamma}\left( 1- \frac{1}{\ord(e_j)}\right)\mathcal{G}_{\mathrm{hyp}}(z, e_j)\big)+\frac{1}{2g_Y}\int\limits_{Y}\mathrm{E}_\Gamma(\zeta)\shyp(\zeta)+ \frac{\mathrm{C_{hyp}}}{8g_Y^{2}}+\frac{2\pi(c_{Y}-1)}{g_Y v_Y}.
\end{align*}
Now, from Lemma \ref{ellipticandhyperbolic}, recall that as $z$ approaches the cusp $p_k$, we have
\begin{align*}
\frac{\mathrm{E}_\Gamma(z)}{2g_Y}+\frac{\mathrm{H}_\Gamma(z)}{2g_Y}=\frac{2\pi}{g_Y}\, \mathcal{C}_{p_k p_k} -\frac{4\pi\log\big(\Im(\sigma_{p_k}^{-1}z)\big)}{g_Y v_Y}-\frac{2\pi}{g_Y v_Y}.
\end{align*}
From \eqref{ghypsum}, we recall that as $z$ approaches the cusp $p_k$, we have
\begin{align*}
&\mathcal{G}_{\mathrm{hyp}}(z, w)=4\pi\mathcal{K}_{p_k}(w) - 
\frac{4\pi}{v_Y}-\frac{4\pi \log\big(\hspace{-0.03cm}\Im({\sigma_{p_k}^{-1}}z)\big)}{v_Y}.
\end{align*}
Then using Proposition \ref{propparabolicintegral}, Lemma \ref{ellipticlemma}, and Lemma \ref{Chyp}, we get
\begin{align}\label{Aw}
A(w)=&4\pi \, \mathcal{K}_{p_k}(w)+ \frac{2\pi}{g_Y} \sum_{\substack{l=1\\ l\not=j}}^{p_\Gamma}\mathcal{C}_{p_lp_j}+ \frac{2\pi}{g_Y}\sum_{j=1}^{e_\Gamma}\left( 1- \frac{1}{\ord(e_j)}\right)\mathcal{K}_{p_k}(e_j)+\frac{2\pi c_{Y}}{g_Y v_Y}+ \delta_Y,
\end{align}
where the absolute value of $\delta_Y$ is bounded by
\begin{align*}
&\frac{2 \pi}{v_Y g_Y}\sum_{j=1}^{e_\Gamma}\left( 1+ \frac{1}{\ord(e_j)}\right)+\frac{2\log 2}{v_Y g_Y}\sum_{j=1}^{e_\Gamma}\big( {\ord(e_j)}+1\big)+\frac{2\pi(d_Y+1)^2}{\lambda_1 v_Y}\\&+\frac{2\pi}{v_Y}+\frac{\log(4 \pi)}{g_Y}
 +\frac{ p_\Gamma}{g_Y v_Y}\bigg(\pi+\frac{4\pi^2}{3}+1\bigg).
\end{align*}
Now, to complete the theorem it remains to compute the following limit
\begin{align*}
\mathcal{G}_{\mathrm{can}}(p_k,p_l)=\lim_{w\to p_l}\big(A(w)-\Phi(w)\big).
\end{align*}
For that we use the formula \eqref{phiz} and \eqref{Aw}. Here note that, by Lemma \ref{FourierexpansionKLF}, as $w$ approaches the cusp $p_l$, we have
\begin{align*}
4\pi\, \mathcal{K}_{p_k}(w)=4\pi\, \mathcal{C}_{p_k p_l}-\frac{4\pi\log\big(\hspace{-0.03cm}\Im(\sigma_{p_l}^{-1}w)\big)}{v_Y}.
\end{align*}
Then using Proposition \ref{propparabolicintegral}, Lemma \ref{ellipticlemma}, and Lemma \ref{Chyp}, we get
\begin{align*}
\mathcal{G}_{\mathrm{can}}(p_k,p_l)=& 4 \pi \, \mathcal{C}_{p_k p_l}+ \frac{2\pi}{g_Y} \sum_{\substack{j=1 \\ j\not= k}}^{p_\Gamma}\mathcal{C}_{p_k p_j}+\frac{2\pi}{g_Y} \sum_{\substack{j=1 \\ j\not= l}}^{p_\Gamma}\mathcal{C}_{p_l p_j}+\frac{4\pi c_Y}{g_Y v_Y}\\
&+\frac{2\pi}{g_Y}\sum_{j=1}^{e_\Gamma} \left( 1- \frac{1}{\ord(e_j)}\right) \left(\mathcal{K}_{p_k}(e_j)
+\mathcal{K}_{p_l}(e_j)\right)+\delta_Y,
\end{align*}
where the absolute value of $\delta_Y$ is bounded by
\begin{align*}
&\frac{4 \pi}{v_Y g_Y}\sum_{j=1}^{e_\Gamma}\left( 1+ \frac{1}{\ord(e_j)}\right)+\frac{4\log 2}{v_Y g_Y}\sum_{j=1}^{e_\Gamma}\big( {\ord(e_j)}+1\big)+\frac{4\pi(d_Y+1)^2}{\lambda_1 v_Y}\\&+\frac{4\pi}{v_Y}+\frac{2\log(4 \pi)}{g_Y}
 +\frac{2 p_\Gamma}{g_Y v_Y}\bigg(\pi+\frac{4\pi^2}{3}+1\bigg).
\end{align*}
This completes the proof.
\end{proof}
\section{Bounds for scattering constants of \texorpdfstring{$\Gamma_0(N)$}{}}\label{BSCGN}
Now, we consider $\Gamma=\Gamma_0(N)$
and $Y=\Gamma_0(N)\backslash \mathbb{H}$ with genus $g_Y$, where $N$ is a positive integer. 
Here we give bounds for scattering constants of $\Gamma$.
\begin{lem}\label{Lemma1Gammagen0N}
Let $\mathcal{C}_{0 \infty}$ denote the scattering constant with respect to the cusps $0$ and  $\infty$. Then
\begin{align*}
8\pi g_Y(1-g_Y)\, \mathcal{C}_{0\infty}= 2g_Y\log N + o(g_Y \log N)\,\ \text{as}\,\ N\to \infty.
\end{align*}
\begin{proof}
By substituting $n=1$ in \eqref{scatteringfunctionGamma0Nainfty} (or substitute $n=N$ in \eqref{phia0Gammagen0N}), we get the scattering function
\begin{align*}
\varphi_{0\infty}(s)
&= \sqrt{\pi}\, \frac{\Gamma(s-1\slash 2)}{\Gamma(s)}\frac{1}{N^s}\frac{\zeta(2s-1)}{\zeta(2s)}\prod_{\substack{p|N\\ p\,\text{prime}}}\frac{p^{2s}-p}{p^{2s}-1}.
\end{align*}
 Then, from the definition of scattering constant (\ref{scatteringconstantdef}) we can write
\begin{align}\label{limitequalto}
\mathcal{C}_{0\infty}=
\lim_{s\to 1}\bigg(\sqrt{\pi}\frac{\Gamma(s-1\slash 2)}{\Gamma(s)}\frac{\zeta(2s-1)}{\zeta(2s)}\frac{1}{N^s}\prod_{\substack{p|N\\ p\,\text{prime}}}\frac{p^{2s}-p}{p^{2s}-1}-\frac{1}{(s-1)v_Y}\bigg).
\end{align}
At $s=1$ we compute the following Taylor expansions 
\begin{align*}
&\frac{1}{N^s}= \frac{1}{N}-\frac{\log N}{N}(s-1)+O\left((s-1)^2\right),\notag\\ 
&\prod_{\substack{p|N\\ p\,\text{prime}}}\frac{p^{2s}-p}{p^{2s}-1}=\prod_{\substack{p|N\\ p\,\text{prime}}}\frac{p}{p+1}+\prod_{\substack{p|N\\ p\,\text{prime}}}\frac{p}{p+1}\sum_{\substack{q|N\\ q\,\text{prime}}}\frac{2q\log q}{q^2-1}(s-1)+O\left((s-1)^2\right).
\end{align*}

\vspace{0.2cm} \noindent
Note that, $\sqrt{\pi}\frac{\Gamma(s-1\slash 2)}{\Gamma(s)}\frac{\zeta(2s-1)}{\zeta(2s)}$ is the scattering function for $\PSLZ$ and from subsection \ref{Eis} we recall that 
\begin{align}\label{nextlemmaainfty}
 &\sqrt{\pi} \frac{\Gamma(s-\frac{1}{2})}{\Gamma(s)} \frac{\zeta(2s-1)}{\zeta(2s)}=\frac{3}{\pi(s-1)}+ \mathcal{C} +O(s-1)\,\ \text{as}\,\ s\to1,
\end{align}
where $\mathcal{C}=\frac{6}{\pi}\big(1-\log (4\pi)-12 \zeta'(-1)\big)$.

\vspace{0.2cm} \noindent
Then, from (\ref{limitequalto}), we get
\begin{align}\label{scateringconstantGamma0infty}
\mathcal{C}_{0\infty}=\frac{1}{v_Y}\bigg(\frac{\pi}{3}\mathcal{C}-\log N+ \sum_{\substack{p|N\\ p\,\text{prime}}}\frac{2p\log p}{p^2-1}\bigg).
\end{align}
Finally, multiplying $8\pi g_Y(1-g_Y)$ of the both side of (\ref{scateringconstantGamma0infty}), we get
\begin{align*}
8\pi g_Y(1-g_Y)\, \mathcal{C}_{0\infty}=\frac{8\pi g_Y(1-g_Y)}{v_{\Gamma_0(N)}}\bigg(\frac{\pi}{3}\mathcal{C}-\log N+ \sum_{\substack{p|N\\ p\,\text{prime}}}\frac{2p\log p}{p^2-1}\bigg).
\end{align*}
Finally, using (\ref{volgeGamma}), and taking into account that $\displaystyle\sum_{\substack{p|N\\ p\,\text{prime}}}\frac{\log p}{p}= O(\log \log N)$ (see e.g., \cite{Bruijn}), we have
\begin{align*}
8\pi g_Y(1-g_Y)\, \mathcal{C}_{0\infty}= 2g_Y\log N + o(g_Y \log N)\,\ \text{as}\,\ N\to \infty.
\end{align*}
This completes the proof.
\end{proof}
\end{lem}
\begin{lem}\label{Lemma2Gammagen0N}
Let $\mathcal{C}_{a\infty}$ denote the scattering constant with respect to the cusps $a=m\slash n$ and $\infty$. Then
\begin{align*}
4\pi (1-g_Y)\sum_{\substack{a\in \mathcal{P}_{\Gamma}\\ a\not=\infty}}\mathcal{C}_{a\infty}= o(g_Y \log N)\,\ \text{as}\,\ N\to \infty.
\end{align*}
\begin{proof}
From the definition (\ref{scatteringconstantdef}), we recall  
\begin{align*}
\mathcal{C}_{a\infty}=\lim_{s\to 1}\left( \varphi_{a\infty}(s)-\frac{1}{(s-1)v_Y}\right),
\end{align*}
where $\varphi_{a\infty}(s)$ is given by the formula \cite{DI82}, p.~247. Means, we can write
\begin{align*}
\mathcal{C}_{a\infty}=\frac{\phi(n)}{\phi\left((n,N\slash n\right))}\lim_{s\to 1}\left(\sqrt{\pi}\frac{\Gamma(s-\frac{1}{2})}{\Gamma(s)}\frac{\zeta(2s-1)}{\zeta(2s)}F(s)-\frac{1}{(s-1)v_Y}\right),
\end{align*}
where
\begin{align*}
F(s)=\left(\frac{(n, N\slash n)}{nN}\right)^s\prod_{\substack{p|N\\ p\,\text{prime}}}\frac{p^{2s}}{p^{2s}-1}\prod_{\substack{q|\frac{N}{n}\\q\, \text{prime}}}\left(1-\frac{1}{q^{2s-1}}\right).
\end{align*}
Now, at $s=1$ we compute the following Taylor expansions 
\begin{align*}
\left(\frac{(n, N\slash n)}{nN}\right)^s=\left(\frac{(n, N\slash n)}{nN}\right)+\left(\frac{(n, N\slash n)}{nN}\right)\log\left(\frac{(n, N\slash n)}{nN}\right)(s-1)
+O\left((s-1)^2\right),\hspace{.8cm}
\end{align*}
\begin{align*}
&\prod_{\substack{p|N\\ p\,\text{prime}}}\frac{p^{2s}}{p^{2s}-1}=\prod_{\substack{p|N\\ p\,\text{prime}}}\frac{p^{2}}{p^{2}-1}-\prod_{\substack{p|N\\ p\,\text{prime}}}\frac{p^{2}}{p^{2}-1}\sum_{\substack{q|N\\q\, \text{prime}}}\frac{2\log q}{q^2-1}(s-1)+O\left((s-1)^2\right),\\
&\prod_{\substack{p|\frac{N}{n}\\ p\,\text{prime}}}\frac{p^{2s}-p}{p^{2s}}=\prod_{\substack{p|\frac{N}{n}\\ p\,\text{prime}}}\frac{p-1}{p}+\prod_{\substack{p|\frac{N}{n}\\ p\,\text{prime}}}\frac{p-1}{p}\sum_{\substack{q|\frac{N}{n}\\q\, \text{prime}}}\frac{2\log q}{q-1}(s-1)+O\left((s-1)^2\right).
\end{align*}
Then, by recalling (\ref{nextlemmaainfty}), we get
\begin{align*}
\mathcal{C}_{a\infty}=\frac{1}{v_Y}\bigg(\frac{\pi}{3}\mathcal{C}+\log\left(\frac{(n, N\slash n)}{nN}\right)-\sum_{\substack{p|N\\ p\,\text{prime}}}\frac{2\log p}{p^2-1}+\sum_{\substack{p|\frac{N}{n}\\ p\,\text{prime}}}\frac{2\log p}{p-1}\bigg).
\end{align*}
Hence from (\ref{numberofcusp}), we get
\begin{align*}
&4\pi (1-g_Y)\sum_{\substack{a\in \mathcal{P}_{\Gamma}\\ a\not=\infty}}\mathcal{C}_{a\infty}\\&= \frac{4\pi (1-g_Y)}{v_Y}\bigg(\frac{\pi}{3}\mathcal{C}-\sum_{\substack{p|N\\ p\,\text{prime}}}\frac{2\log p}{p^2-1}\bigg)\sum_{\substack{n|N\\n\not=1}}\phi\left(\left(n, N\slash n\right)\right)\notag\\
&+\frac{4\pi (1-g_Y)}{v_Y}\sum_{\substack{n|N\\n\not=1}}\phi\left(\left(n, N\slash n\right)\right)\bigg(\log\left(\frac{(n, N\slash n)}{nN}\right)+\sum_{\substack{p|\frac{N}{n}\\ p\,\text{prime}}}\frac{2\log p}{p-1}\bigg).
\end{align*}
Finally, using (\ref{volgeGamma}), and taking into account that $\displaystyle\sum_{\substack{p|N\\ p\,\text{prime}}}\frac{\log p}{p}= O(\log \log N)$, we have
\begin{align*}
4\pi (1-g_Y)\sum_{\substack{a\in \mathcal{P}_{\Gamma}\\ a\not=\infty}}\mathcal{C}_{a\infty}= o(g_Y \log N)\,\ \text{as}\,\ N\to \infty.
\end{align*}
This completes the proof.
\end{proof}
\end{lem}
\begin{lem}\label{Lemma3Gammagen0N}
Let $\mathcal{C}_{a0}$ denote the scattering constant with respect to the cusps $a=m\slash n$ and  $0$. Then
\begin{align*}
4\pi (1-g_Y)\sum_{\substack{a\in \mathcal{P}_{\Gamma}\\ a\not=0}}\mathcal{C}_{a0}= o(g_Y \log N)\,\ \text{as}\,\ N\to \infty.
\end{align*}
\begin{proof}
Recalling the definition of the scattering constant \eqref{scatteringconstantdef}, we can write 
\begin{align*}
\mathcal{C}_{a0}=\lim_{s\to 1}\bigg( \varphi_{a0}(s)-\frac{1}{(s-1)v_Y}\bigg),
\end{align*}
where $\varphi_{a0}(s)$ is the scattering function  with respect to the cusps $a=m\slash n$ and $0$ (see \eqref{phia0Gammagen0N})
\begin{align*}
\varphi_{a0}(s)=\sqrt{\pi}\frac{\Gamma(s-\frac{1}{2})}{\Gamma(s)}\frac{\zeta(2s-1)}{\zeta(2s)}\frac{\phi(N\slash n)}{\phi(n, N\slash n)}\,G(s),
\end{align*}
where
\begin{align*}
G(s)=\frac{(n^2,N)^s}{N^{2s}}\prod_{\substack{p|N\\ p\,\text{prime}}}\frac{p^{2s}}{p^{2s}-1}\prod_{\substack{q|n\\ q\,\text{prime} }}\left(1-\frac{1}{q^{2s-1}}\right).
\end{align*}
Hence we have
\begin{align*}
\mathcal{C}_{a0}=\frac{\phi(N\slash n)}{\phi(n, N\slash n)}\lim_{s\to 1}\bigg( \sqrt{\pi}\frac{\Gamma(s-\frac{1}{2})}{\Gamma(s)}\frac{\zeta(2s-1)}{\zeta(2s)}\,G(s)-\frac{1}{(s-1)v_Y}\bigg).
\end{align*}

At $s=1$ we compute the following Taylor expansions 
\begin{align*}
&\frac{(n^2,N)^s}{N^{2s}}=\frac{(n^2,N)}{N^{2}}+\frac{(n^2,N)}{N^{2}}\log \bigg(\frac{(n^2,N)}{N^{2}}\bigg)(s-1)+ O\left((s-1)^2\right),\\\\
&\prod_{\substack{p|N\\ p\,\text{prime}}}\frac{p^{2s}}{p^{2s}-1}=\prod_{\substack{p|N\\ p\,\text{prime}}}\frac{p^{2}}{p^{2}-1}-\prod_{\substack{p|N\\ p\,\text{prime}}}\frac{p^{2}}{p^{2}-1}\sum_{\substack{q|N\\q\, \text{prime}}}\frac{2\log q}{q^2-1}(s-1)+O\left((s-1)^2\right),\\\\
&\prod_{\substack{q|n\\ q\,\text{prime} }}\bigg(1-\frac{1}{q^{2s-1}}\bigg)=\prod_{\substack{q|n\\ q\,\text{prime} }}\bigg(1-\frac{1}{q}\bigg)+\prod_{\substack{q|n\\ q\,\text{prime} }}\bigg(1-\frac{1}{q}\bigg)\sum_{\substack{p|N\\ p\,\text{prime}}}\frac{2\log p}{p-1}(s-1)+O\left((s-1)^2\right).
\end{align*}
Then, by recalling (\ref{nextlemmaainfty}), we get
\begin{align*}
\mathcal{C}_{a0}=\frac{1}{v_Y}\bigg(\frac{\pi}{3}\mathcal{C}+\log \bigg(\frac{(n^2,N)}{N^2}\bigg)-\sum_{\substack{p|N\\ p\,\text{prime}}}\frac{2\log p}{p^2-1}+\sum_{\substack{p|n\\ p\,\text{prime}}}\frac{2\log p}{p-1}\bigg).
\end{align*}
Then, from (\ref{numberofcusp}), we get
\begin{align*}
&4\pi (1-g_Y)\sum_{\substack{a\in \mathcal{P}_{\Gamma}\\ a\not=0}}\mathcal{C}_{a0}\\&=\frac{4 \pi (1-g_Y)}{v_Y}\bigg(\frac{\pi}{3}\mathcal{C}-\sum_{\substack{p|N\\ p\,\text{prime}}}\frac{2\log p}{p^2-1}\bigg)\sum_{\substack{n|N\\n\not=N}}\phi\left(\left(n, N\slash n\right)\right)\notag\\
&+\frac{4\pi (1-g_Y)}{v_Y}\sum_{\substack{n|N\\n\not=N}}\phi\left(\left(n, N\slash n\right)\right)\bigg(\log \bigg(\frac{(n^2,N)}{N^2}\bigg)+\sum_{\substack{p|n\\ p\,\text{prime}}}\frac{2\log p}{p-1}\bigg).
\end{align*}
Finally, using (\ref{volgeGamma}), and taking into account that $\displaystyle\sum_{\substack{p|N\\ p\,\text{prime}}}\frac{\log p}{p}= O(\log \log N)$, we have
\begin{align*}
4\pi (1-g_Y)\sum_{\substack{a\in \mathcal{P}_{\Gamma}\\ a\not=0}}\mathcal{C}_{a0}= o(g_Y \log N)\,\ \text{as}\,\ N\to \infty.
\end{align*}
This completes the proof.
\end{proof}
\end{lem}
\section{Bounds for Kronecker limit functions of \texorpdfstring{$\Gamma_0(N)$}{}}\label{BKLFGN}
 Here we prove a relation between the Kronecker limit functions of $\Gamma=\Gamma_0(N)$ and $\PSLZ$. For notational convenience, in this section by $\mathcal{K}_\infty(z)_{\Gamma(1)}$ resp.~${E}_\infty(z,s)_{\Gamma(1)}$ we denote the Kronecker limit function resp.~the Eisenstein series with respect to the cusp $\infty$ of the group $\Gamma(1)=\PSLZ$.
Then we derive bounds for Kronecker limit functions of $\Gamma_0(N)$. 
\begin{lem}\label{lemma3N}
Let $\mathcal{K}_{0}(z)$ and $\mathcal{K}_{\infty}(z)$ denote the Kronecker limit functions with respect to the cusps $0$ and $\infty$ of the group $\Gamma=\Gamma_0(N)$. Let $\mathcal{K}_\infty(z)_{\Gamma(1)}$ denote the Kronecker limit function with respect to the cusp $\infty$ of the group $\Gamma(1)=\PSLZ$. Then 
\begin{align*}
&\mathcal{K}_{0}(z)+\mathcal{K}_{\infty}(z)+\frac{2}{v_Y}\bigg(\log N-\sum_{\substack{p|N\\ p\,\mathrm{prime}}}\frac{\log p}{p+1}\bigg)\notag\\
&=\frac{1}{N}\prod_{\substack{p|N\\ p\,\mathrm{prime}}}\frac{p^2}{p^2-1}\sum_{\substack{d|N\\d>0}} \frac{\mu(d)}{d}\bigg(\mathcal{K}_\infty(Nz\slash d)_{\Gamma(1)}+ \mathcal{K}_\infty(d z)_{\Gamma(1)}\bigg),
\end{align*}
where $\mu(d)$ is the M\"obius function.
\begin{proof}
From \cite{GZ}, p.~240, we know the following formula
\begin{align}\label{Eisensteinrelation}
{E}_\infty(z,s)= N^{-s}\prod_{\substack{p|N\\ p\,\text{prime}}}\frac{p^{2s}}{p^{2s}-1} \sum_{\substack{d|N\\d>0}}\frac{\mu(d)}{d^s}{E}_\infty\big(Nz\slash d,s\big)_{\Gamma(1)},
\end{align}
where in the left hand side ${E}_\infty(z,s)$ denotes the Eisenstein series for $\Gamma$ with respect to the cusp $\infty$ and in the right hand side $E_\infty(Nz\slash d,s)_{\Gamma(1)}$ denotes the Eisenstein series for $\Gamma(1)$ with respect to the cusp $\infty$.

\vspace{0.2cm} \noindent
Now, using formula (\ref{Eisensteinrelation}), we can write
\begin{align}\label{KLFN}
\mathcal{K}_{\infty}(z)=\lim_{s\to 1}\bigg( N^{-s}\prod_{\substack{p|N\\ p\,\text{prime}}}\frac{p^{2s}}{p^{2s}-1} \sum_{\substack{d|N\\d>0}}\frac{\mu(d)}{d^s}{E}_\infty\big(Nz\slash d,s\big)_{\Gamma(1)}-\frac{1}{(s-1)v_Y}\bigg).
\end{align}
To compute this limit (\ref{KLFN}), we use the following expansions.
\begin{align*}
{E}_\infty\big(Nz\slash d,s\big)_{\Gamma(1)}&= \frac{3}{\pi(s-1)}+ \mathcal{K}_\infty(Nz \slash d)_{\Gamma(1)}+O(s-1)\,\ \text{as}\,\ s\to 1,
\end{align*} 
which is the well-known Laurent expansion of the Eisenstein series. 

\vspace{0.2cm} \noindent
At $s=1$ we compute the following Taylor series expansions:
\begin{align*}
&\frac{\mu(d)}{d^s}= \frac{\mu(d)}{d}-\frac{\mu(d) \log d}{d}(s-1)+O\big((s-1)^2\big),\notag\\
&\prod_{\substack{p|N\\ p\,\text{prime}}}\frac{p^{2s}}{p^{2s}-1}= \prod_{\substack{p|N\\ p\,\text{prime}}}\frac{p^2}{p^2-1}-\prod_{\substack{p|N\\ p\,\text{prime}}}\frac{p^2}{p^2-1}\sum_{\substack{q|N\\ q\, \text{prime}}}\frac{2\log q}{q^2-1}(s-1)+O((s-1)^2),\\
&\frac{1}{N^s}= \frac{1}{N}-\frac{\log N}{N}(s-1)+O\left((s-1)^2\right).
\end{align*}
Then from (\ref{KLFN}), we get
\begin{align}\label{KLFinftyN}
&\mathcal{K}_{\infty}(z)+\frac{1}{v_Y}\bigg(\log N-\sum_{\substack{p|N\\ p\,\text{prime}}}\frac{\log p}{p+1}\bigg)=\frac{1}{N}\prod_{\substack{p|N\\ p\,\text{prime}}}\frac{p^2}{p^2-1}\sum_{\substack{d|N\\d>0}} \frac{\mu(d)}{d}\mathcal{K}_\infty(Nz\slash d)_{\Gamma(1)}.
\end{align}
Now, to compute $\mathcal{K}_{0}(z)$, we consider a scaling matrix of the cusp $0$ as
\begin{align*}
\sigma_0=\bigg( \begin{array}{ccc} 0 &-{1}\slash{\sqrt{N}}\\
 \sqrt{N} & 0  \end{array}\bigg).
 \end{align*}
Then using the Fricke involution, we get 
\begin{align*}
{E}_0(z, s)= {E}_\infty(\sigma_0^{-1}z, s).
\end{align*}
Then we have
\begin{align}\label{KLF0N}
&\mathcal{K}_{0}(z)= \mathcal{K}_{\infty}(\sigma_0^{-1}z)=\mathcal{K}_{\infty}\big({-1}\slash {Nz}\big)\notag\\
&= \frac{1}{N}\prod_{\substack{p|N\\ p\,\text{prime}}}\frac{p^2}{p^2-1}\sum_{\substack{d|N\\d>0}} \frac{\mu(d)}{d}\mathcal{K}_\infty(d z)_{\Gamma(1)}-\frac{1}{v_Y}\bigg(\log N-\sum_{\substack{p|N\\ p\, \text{prime}}}\frac{\log p}{p+1}\bigg).
\end{align}
Finally, combining (\ref{KLFinftyN}) and (\ref{KLF0N}), we can write
\begin{align*}
&\mathcal{K}_{0}(z)+\mathcal{K}_{\infty}(z)+\frac{2}{v_Y}\bigg(\log N-\sum_{\substack{p|N\\ p\,\text{prime}}}\frac{\log p}{p+1}\bigg)\notag\\
&=\frac{1}{N}\prod_{\substack{p|N\\ p\,\text{prime}}}\frac{p^2}{p^2-1}\sum_{\substack{d|N\\d>0}} \frac{\mu(d)}{d}\bigg(\mathcal{K}_\infty(Nz\slash d)_{\Gamma(1)}+ \mathcal{K}_\infty(d z)_{\Gamma(1)}\bigg).
\end{align*}
This completes the proof.
\end{proof}
\end{lem}
\begin{lem}\label{lemma3N21}
 Let $\mathcal{K}_\infty(z)_{\Gamma(1)}$ denote the Kronecker limit function with respect to the cusp $\infty$ of the group $\Gamma(1)=\PSLZ$. Let $\nu_2$ denote the number of elliptic fixed points of $\Gamma$ with $\ord(e_j)= 2$. Then
\begin{align*}
&\frac{(1-g_Y)}{N}\sum_{j=1}^{\nu_2} \left( 1- \frac{1}{\ord(e_j)}\right)\prod_{\substack{p|N\\ p\,\mathrm{prime}}}\frac{p^2}{p^2-1}\sum_{\substack{d|N\\d>0}} \frac{\mu(d)}{d}\bigg(\mathcal{K}_\infty(Ne_j\slash d)_{\Gamma(1)}+ \mathcal{K}_\infty(d e_j)_{\Gamma(1)}\bigg)\notag\\
&=o(g_Y \log N)\,\ \text{as}\,\ N \to \infty,
\end{align*}
where $\mu(d)$ is the M\"obius function.
\begin{proof}
From (\ref{nfori}), we recall that if $e_j$ be an elliptic fixed point of $\Gamma$ with $\ord(e_j)=2$, then
\begin{align*}
e_j= \frac{n+i}{n^2+1}\,\ \text{for}\,\ n=0,\dotsc, N-1 \,\ \text{such that}\,\ n^2+1 \equiv 0 \,(\text{mod}\, N).
\end{align*}
By substituting these values of $e_j$ in the term $\mathcal{K}_\infty(Ne_j\slash d)_{\Gamma(1)}+ \mathcal{K}_\infty(d e_j)_{\Gamma(1)}$ we get our desired asymptotics. 
As recalled in subsection \ref{Eis}, we have  
\begin{align}\label{KLFPSL}
&\mathcal{K}_\infty(z)_{\Gamma(1)}=-\frac{1}{2\pi}\log(|\Delta(z)|y^{6} )+ \mathcal{C},
\end{align}
where $z=x+iy$ and $\Delta(z)$ denotes the modular discriminant, and $\mathcal{C}$ denotes the scattering constant \eqref{scatteringconstantPSLZ}. Using the Fourier expansion 
\begin{align*}
\Delta(z)&= \sum_{n=1}^\infty \tau(n)e^{2\pi i n z},
\end{align*} 
for any positive integer $d$, we derive the bound
\begin{align}
\log|\Delta(dz)|&\leq - 2\pi d y+ \log\bigg(1+ \sum_{n=1}^\infty |\tau(n+1)|e^{-2\pi  n d y}\bigg)\label{logdelta}.
\end{align}

\vspace{0.2cm} \noindent
When $e_j$ is an elliptic fixed point of $\Gamma$ with $\ord(e_j)=2$, then using (\ref{nfori}) and (\ref{KLFPSL}), we get
\begin{align}\label{iN}
\mathcal{K}_\infty(Ne_j\slash d)_{\Gamma(1)}+ \mathcal{K}_\infty(d e_j)_{\Gamma(1)}=& \frac{-3}{\pi}\log\bigg(\frac{N}{(n^2+1)^2}\bigg)-\frac{1}{2\pi}\log \bigg|\Delta\bigg(\frac{N(n+i)}{d(n^2+1)}\bigg)\bigg|\notag\\&-\frac{1}{2\pi}\log \bigg|\Delta\bigg(\frac{nd+id}{n^2+1}\bigg)\bigg|
+2\mathcal{C},
\end{align}
where $n$ satisfies the conditions given in (\ref{nfori}), $d|N\,(\text{with}\, d>0)$. 

\vspace{0.2cm} \noindent
When $n=0$ in (\ref{iN}), we get
\begin{align*}
&\mathcal{K}_\infty(Ne_j\slash d)_{\Gamma(1)}+ \mathcal{K}_\infty(d e_j)_{\Gamma(1)}= \frac{-3}{\pi}\log({N})-\frac{1}{2\pi}\log \big(|\Delta(Ni\slash d)\Delta(di)|\big)+2 \mathcal{C}.
\end{align*} 
Now, note that using (\ref{logdelta}), we can write 
\begin{align*}
\log \big(|\Delta(Ni\slash d)|\big)=O(N \slash d),\,\ \,\ \log \big(|\Delta(di)|\big)=O(d).
\end{align*}
This implies, there exist some constants $c_1, c_2$ independent of $N$, $d$, and $j$, such that 
\begin{align*}
\mathcal{K}_\infty(Ne_j\slash d)_{\Gamma(1)}+ \mathcal{K}_\infty(d e_j)_{\Gamma(1)}= \frac{c_1N}{d}+ c_2 d+ o(N)\,\ \text{as}\,\ N\to \infty.
\end{align*}

\noindent
For $n=1,\dotsc, N-1 \,\ \text{such that}\,\ n^2+1 \equiv 0 \,(\text{mod}\, N)$ in (\ref{iN}), we get
\begin{align*}
&\mathcal{K}_\infty(Ne_j\slash d)_{\Gamma(1)}+ \mathcal{K}_\infty(d e_j)_{\Gamma(1)}= \frac{c_3N}{d(n^2+1)}+\frac{c_4 d}{n^2+1} + o(N)\,\ \text{as}\,\ N\to \infty,
\end{align*}
where the constants $c_3, c_4$ are independent of $N, d, n$, and $j$.

\vspace{0.2cm} \noindent
Now, using the fact that $\displaystyle\sum_{\substack{d|N \\ d>0}}\frac{\mu(d)}{d}=\prod_{\substack{p|N\\ p\,\text{prime}}}\frac{p-1}{p}$, we can write
\begin{align}\label{halfbound}
&\frac{1}{N}\sum_{j=1}^{\nu_2} \left( 1- \frac{1}{\ord(e_j)}\right)\prod_{\substack{p|N\\ p\,\text{prime}}}\frac{p^2}{p^2-1}\sum_{\substack{d|N\\d>0}} \frac{\mu(d)}{d}\bigg(\mathcal{K}_\infty(Ne_j\slash d)_{\Gamma(1)}+ \mathcal{K}_\infty(d e_j)_{\Gamma(1)}\bigg)\notag\\
&=\frac{1}{2N}\prod_{\substack{p|N\\ p\,\text{prime}}}\frac{p^2}{p^2-1}\sum_{\substack{d|N\\d>0}} \frac{\mu(d)}{d}\sum_{j=1}^{\nu_2} \bigg(\mathcal{K}_\infty(Ne_j\slash d)_{\Gamma(1)}+ \mathcal{K}_\infty(d e_j)_{\Gamma(1)}\bigg)\notag\\
&=\frac{1}{2N}\prod_{\substack{p|N\\ p\,\text{prime}}}\frac{p^2}{p^2-1}\sum_{\substack{d|N\\d>0}} \frac{\mu(d)}{d}\bigg(\frac{c_1N}{d}+c_2 d\bigg)\notag\\
&\hspace{.4cm}+\frac{1}{2N}\prod_{\substack{p|N\\ p\,\text{prime}}}\frac{p^2}{p^2-1}\sum_{\substack{d|N\\d>0}}\frac{\mu(d)}{d} \sum_{n=1}^{N-1}\bigg(\frac{c_3N}{d(n^2+1)}+\frac{c_4 d}{n^2+1}\bigg)+o(N)\notag\\
&=\frac{c_1}{2}+c_3+o(N)\,\ \text{as}\,\ N\to \infty.
\end{align}
Note that, in the last inequality of (\ref{halfbound}) we have used a well-known property of the M\"obius function which is $\displaystyle\sum_{\substack{d|N\\ d>0}}\frac{\mu(d)}{d^2}=\prod_{\substack{p|N\\ p\,\text{prime}}}\frac{p^2-1}{p^2}$. Then using (\ref{halfbound}), we get
\begin{align*}
&\frac{(1-g_Y)}{N}\sum_{j=1}^{\nu_2} \left( 1- \frac{1}{\ord(e_j)}\right)\prod_{\substack{p|N\\ p\,\text{prime}}}\frac{p^2}{p^2-1}\sum_{\substack{d|N\\d>0}} \frac{\mu(d)}{d}\bigg(\mathcal{K}_\infty(Ne_j\slash d)_{\Gamma(1)}+ \mathcal{K}_\infty(d e_j)_{\Gamma(1)}\bigg)\notag\\
&=o(g_Y \log N)\,\ \text{as}\,\ N \to \infty.
\end{align*}
This completes the proof.
\end{proof}
\end{lem}
\begin{lem}\label{lemma3N22}
 Let $\mathcal{K}_\infty(z)_{\Gamma(1)}$ denote the Kronecker limit function with respect to the cusp $\infty$ of the group $\Gamma(1)=\PSLZ$. Let $\nu_3$ denote the number of elliptic fixed points of $\Gamma$ with $\ord(e_j)= 3$. Then
\begin{align*}
&\frac{(1-g_Y)}{N}\sum_{j=1}^{\nu_3} \left( 1- \frac{1}{\ord(e_j)}\right)\prod_{\substack{p|N\\ p\,\mathrm{prime}}}\frac{p^2}{p^2-1}\sum_{\substack{d|N\\d>0}} \frac{\mu(d)}{d}\bigg(\mathcal{K}_\infty(Ne_j\slash d)_{\Gamma(1)}+ \mathcal{K}_\infty(d e_j)_{\Gamma(1)}\bigg)\notag\\
&=o(g_Y \log N)\,\ \text{as}\,\ N \to \infty,
\end{align*}
where $\mu(d)$ is the M\"obius function.
\begin{proof}
The proof is similar to the proof of Lemma \ref{lemma3N21}. From (\ref{nforother}), we recall that if $e_j$ be an elliptic fixed point of $\Gamma$ with $\ord(e_j)=3$, then
\begin{align*}
e_j= \frac{n+\frac{1+i\sqrt{3}}{2}}{n^2-n+1}\,\ \text{for}\,\ n=0,\dotsc, N-1 \,\ \text{such that}\,\ n^2-n+1 \equiv 0 \,(\text{mod}\, N).
\end{align*}
By substituting these values of $e_j$ in the term $\mathcal{K}_\infty(Ne_j\slash d)_{\Gamma(1)}+ \mathcal{K}_\infty(d e_j)_{\Gamma(1)}$, we get
\begin{align}\label{otherellipticN}
&\mathcal{K}_\infty(Ne_j\slash d)_{\Gamma(1)}+ \mathcal{K}_\infty(d e_j)_{\Gamma(1)}\notag\\
&= \frac{-3}{\pi}\log\bigg(\frac{3N}{(2n^2-2n+2)^2}\bigg)-\frac{1}{2\pi}\log \bigg|\Delta\bigg(\frac{N(2n+1+i\sqrt{3})}{2d(n^2-n+1)}\bigg)\bigg|\notag\\& \hspace{.4cm}-\frac{1}{2\pi}\log \bigg|\Delta\bigg(\frac{d(2n+1+i\sqrt{3})}{2(n^2-n+1)}\bigg)\bigg|
+2\mathcal{C},
\end{align}
where $n$ satisfies the conditions given in (\ref{nforother}), $d|N$ and $\mathcal{C}$ is the scattering constant \eqref{scatteringconstantPSLZ}. 

\vspace{0.2cm} \noindent
For $n=0$ in (\ref{otherellipticN}), we get
\begin{align*}
\mathcal{K}_\infty(Ne_j\slash d)_{\Gamma(1)}+ \mathcal{K}_\infty(d e_j)_{\Gamma(1)}=\frac{-3}{\pi}\log\bigg(\frac{3N}{4}\bigg)&-\frac{1}{2\pi}\log \bigg|\Delta\bigg(\frac{N\sqrt{3}i}{2d}\bigg)\Delta\bigg(\frac{d\sqrt{3}i}{2}\bigg)\bigg|+2\mathcal{C}.
\end{align*}
Then, using (\ref{logdelta}), we can write
\begin{align*}
&\log \bigg|\Delta\bigg(\frac{N\sqrt{3}i}{2d}\bigg)\bigg|=O(N\slash d),\,\ \,\ \log \bigg|\Delta\bigg(\frac{d\sqrt{3}i}{2}\bigg)\bigg|=O(d).
\end{align*}
This implies, there exist some constants $c'_1, c'_2$, independent of $N, d$ and $j$, such that 
\begin{align*}
&\mathcal{K}_\infty(Ne_j\slash d)_{\Gamma(1)}+ \mathcal{K}_\infty(d e_j)_{\Gamma(1)}=\frac{c'_1N}{d}+c'_2 d + o(N)\,\ \text{as}\,\ N\to \infty.
\end{align*}
When $n=1,\dotsc, N-1 \,\ \text{such that}\,\ n^2-n+1 \equiv 0 \,(\text{mod}\, N)$ in (\ref{otherellipticN}), we get
\begin{align*}
&\mathcal{K}_\infty(Ne_j\slash d)_{\Gamma(1)}+ \mathcal{K}_\infty(d e_j)_{\Gamma(1)}= \frac{c'_3N}{d(n^2-n+1)}+\frac{c'_4 d}{(n^2-n+1)} + o(N)\,\ \text{as}\,\ N\to \infty,
\end{align*}
where the constants $c'_3, c'_4$ are independent of $N, d, n$ and $j$.

\vspace{0.2cm} \noindent
Now, using the fact that $\displaystyle\sum_{\substack{d|N\\ d>0}}\frac{\mu(d)}{d}=\prod_{\substack{p|N\\ p\,\text{prime}}}\frac{p-1}{p}$, we can write
\begin{align}\label{otherhalfbound}
&\frac{1}{N}\sum_{j=1}^{\nu_3} \left( 1- \frac{1}{\ord(e_j)}\right)\prod_{\substack{p|N\\ p\,\text{prime}}}\frac{p^2}{p^2-1}\sum_{\substack{d|N\\d>0}} \frac{\mu(d)}{d}\bigg(\mathcal{K}_\infty(Ne_j\slash d)_{\Gamma(1)}+ \mathcal{K}_\infty(d e_j)_{\Gamma(1)}\bigg)\notag\\
&=\frac{2}{3N}\prod_{\substack{p|N\\ p\,\text{prime}}}\frac{p^2}{p^2-1}\sum_{\substack{d|N\\d>0}} \frac{\mu(d)}{d}\sum_{j=1}^{\nu_3} \bigg(\mathcal{K}_\infty(Ne_j\slash d)_{\Gamma(1)}+ \mathcal{K}_\infty(d e_j)_{\Gamma(1)}\bigg)\notag\\
&=\frac{2}{3N}\prod_{\substack{p|N\\ p\,\text{prime}}}\frac{p^2}{p^2-1}\sum_{\substack{d|N\\d>0}} \frac{\mu(d)}{d}\bigg(\frac{c'_1N}{d}+c'_2 d\bigg)\notag\\
&\hspace{.4cm}+\frac{2}{3N}\prod_{\substack{p|N\\ p\,\text{prime}}}\frac{p^2}{p^2-1}\sum_{\substack{d|N\\d>0}}\frac{\mu(d)}{d} \sum_{n=1}^{N-1}\bigg(\frac{c'_3N}{d(n^2-n+1)}+\frac{c'_4 d}{n^2-n+1}\bigg)+o(N)\notag\\
&=\frac{c'_1}{2}+c'_3+o(N)\,\ \text{as}\,\ N\to \infty.
\end{align}
Note that in the last inequality of (\ref{otherhalfbound}) we have used a well-known property of the M\"obius function which is $\displaystyle\sum_{\substack{d|N\\d>0}}\frac{\mu(d)}{d^2}=\prod_{\substack{p|N\\ p\,\text{prime}}}\frac{p^2-1}{p^2}$. Then using (\ref{halfbound}), we get
\begin{align*}
&\frac{(1-g_Y)}{N}\sum_{j=1}^{\nu_3} \left( 1- \frac{1}{\ord(e_j)}\right)\prod_{\substack{p|N\\ p\,\text{prime}}}\frac{p^2}{p^2-1}\sum_{\substack{d|N\\d>0}} \frac{\mu(d)}{d}\bigg(\mathcal{K}_\infty(Ne_j\slash d)_{\Gamma(1)}+ \mathcal{K}_\infty(d e_j)_{\Gamma(1)}\bigg)\notag\\
&=o(g_Y \log N)\,\ \text{as}\,\ N \to \infty.
\end{align*}
This completes the proof.
\end{proof}
\end{lem}
\begin{prop}\label{prop3N}Let $\mathcal{K}_{0}(z)$ and $\mathcal{K}_{\infty}(z)$ denote the Kronecker limit functions for the group $\Gamma=\Gamma_0(N)$ with respect to the cusps $0$ and $\infty$ respectively. Let $\{e_j\}_{j=1}^{e_{\Gamma}}$ be the set of elliptic fixed points of $\Gamma$. Then
\begin{align*}
4\pi(1-g_Y) \sum_{j=1}^{e_{\Gamma}} \left( 1- \frac{1}{\ord(e_j)}\right) \big(\mathcal{K}_{0}(e_j)+\mathcal{K}_{\infty}(e_j)\big) = o(g_Y \log N)\,\ \text{as}\,\ N\to \infty.
\end{align*}
\begin{proof}
Using Lemma \ref{lemma3N}, we can write
\begin{align}\label{twosummened}
& 4\pi(1-g_Y)\sum_{j=1}^{e_{\Gamma}} \left( 1- \frac{1}{\ord(e_j)}\right) \big(\mathcal{K}_{0}(e_j)+\mathcal{K}_{\infty}(e_j)\big)\notag\\
&= \frac{4\pi(1-g_Y)}{N}\sum_{j=1}^{e_{\Gamma}} \left( 1- \frac{1}{\ord(e_j)}\right)\prod_{\substack{p|N\\ p\,\text{prime}}}\frac{p^2}{p^2-1}\sum_{\substack{d|N\\d>0}} \frac{\mu(d)\big(\mathcal{K}_\infty(Ne_j\slash d)_{\Gamma(1)}+ \mathcal{K}_\infty(d e_j)_{\Gamma(1)}\big)}{d}\notag\\
&\hspace{4.6cm}+\frac{8\pi(1-g_Y)}{v_Y}\sum_{j=1}^{e_{\Gamma}} \left( 1- \frac{1}{\ord(e_j)}\right)\bigg(\sum_{\substack{p|N\\ p\,\text{prime}}}\frac{\log p}{p+1}-\log N\bigg).
\end{align}
Using \eqref{volgeGamma}, \eqref{ellipticboundGammaN}, and taking into account the estimate $\displaystyle\sum_{\substack{p|N\\ p\, \text{prime}}}\frac{\log p}{p}= O(\log \log N),$
we get 
\begin{align*}
&\frac{8\pi(1-g_Y)}{v_Y}\sum_{j=1}^{e_{\Gamma}}\left( 1- \frac{1}{\ord(e_j)}\right) \bigg(\sum_{\substack{p|N\\ p\,\text{prime}}}\frac{\log p}{p+1}-\log N\bigg)=o(g_Y\log N)\,\ \text{as}\,\ N\to \infty,
\end{align*}
which is the last line of (\ref{twosummened}). 

\noindent
To complete the proof it suffices to show that
\begin{align*}
&\frac{4\pi(1-g_Y)}{N}\sum_{j=1}^{e_{\Gamma}} \left( 1- \frac{1}{\ord(e_j)}\right)\prod_{\substack{p|N\\ p\,\text{prime}}}\frac{p^2}{p^2-1}\sum_{\substack{d|N\\d>0}} \frac{\mu(d)\big(\mathcal{K}_\infty(Ne_j\slash d)_{\Gamma(1)}+ \mathcal{K}_\infty(d e_j)_{\Gamma(1)}\big)}{d}\notag\\
&=o(g_Y\log N)\,\ \text{as}\,\ N\to \infty.
\end{align*}
Now, note that
\begin{align*}
&\frac{4\pi(1-g_Y)}{N}\sum_{j=1}^{e_{\Gamma}} \left( 1- \frac{1}{\ord(e_j)}\right)\prod_{\substack{p|N\\ p\,\text{prime}}}\frac{p^2}{p^2-1}\sum_{\substack{d|N\\d>0}} \frac{\mu(d)\big(\mathcal{K}_\infty(Ne_j\slash d)_{\Gamma(1)}+ \mathcal{K}_\infty(d e_j)_{\Gamma(1)}\big)}{d}\notag\\
=&\frac{(1-g_Y)}{N}\sum_{j=1}^{\nu_2} \left( 1- \frac{1}{\ord(e_j)}\right)\prod_{\substack{p|N\\ p\,\text{prime}}}\frac{p^2}{p^2-1}\sum_{\substack{d|N\\d>0}} \frac{\mu(d)}{d}\bigg(\mathcal{K}_\infty(Ne_j\slash d)_{\Gamma(1)}+ \mathcal{K}_\infty(d e_j)_{\Gamma(1)}\bigg)\notag\\
&+\frac{(1-g_Y)}{N}\sum_{j=1}^{\nu_3} \left( 1- \frac{1}{\ord(e_j)}\right)\prod_{\substack{p|N\\ p\,\text{prime}}}\frac{p^2}{p^2-1}\sum_{\substack{d|N\\d>0}} \frac{\mu(d)}{d}\bigg(\mathcal{K}_\infty(Ne_j\slash d)_{\Gamma(1)}+ \mathcal{K}_\infty(d e_j)_{\Gamma(1)}\bigg).
\end{align*}
Finally, Lemma \ref{lemma3N21} and Lemma \ref{lemma3N22} completes the proof.
\end{proof}
\end{prop}
\section{Proof of Theorem \ref{maintheorem2article}}\label{section8}
Finally, combining Theorem \ref{maintheorem1article} with the results from section \ref{BSCGN} and \ref{BKLFGN} we prove Theorem \ref{maintheorem2article}. 
In this section, we consider $\Gamma=\Gamma_0(N)$ and $Y=\Gamma_0(N)\backslash \mathbb{H}$ with compactification $X_0(N)=\overline{\Gamma\backslash \mathbb{H}}$, where $N$ is a positive integer. By $g_Y$ we denote the genus of $Y$.
We derive an asymptotic expression for the canonical Green's function evaluated at the cusps $0$ and $\infty$ of $X_0(N)$.
\begin{proof}[Proof of Theorem \ref{maintheorem2article}]
From Theorem \ref{maintheorem1article}, we have
\begin{align}\label{fromtheorem}
2g_Y(1-g_Y)\, \mathcal{G}_{\mathrm{can}}(0, \infty)&=8 \pi g_Y(1-g_Y)\, \mathcal{C}_{0\infty}+ 4\pi(1-g_Y)\bigg( \sum_{\substack{a\in \mathcal{P}_{\small{\Gamma}} \\ a\not=\infty}}\mathcal{C}_{a \infty}+\sum_{\substack{a\in \mathcal{P}_{\small{\Gamma}} \\ a\not=0}}\mathcal{C}_{a0}\bigg)\notag \\
&\phantom{=}+4\pi(1-g_Y)\sum_{j=1}^{e_{\Gamma}} \left( 1- \frac{1}{\ord(e_j)}\right)\big(\mathcal{K}_{0}(e_j)
\phantom{=}+\mathcal{K}_{\infty}(e_j)\big)\notag \\&+\frac{8\pi(1-g)c_{Y}}{v_Y}+\delta_{Y},
\end{align}
where the absolute value of $\delta_{Y}$ is bounded by
\begin{align}\label{deltaX}
\frac{8\pi(g_Y-1)}{v_Y}\sum_{j=1}^{e_{\Gamma}} \left( 1+ \frac{1}{\ord(e_j)}\right)+\frac{8(g_Y-1)}{v_Y}\sum_{j=1}^{e_{\Gamma}}(1+\ord(e_j))+\frac{8\pi g_Y(g_Y-1)(d_{Y}+1)^2}{\lambda_1v_{\small{\Gamma}}}\notag\\+\frac{8\pi g_Y(g_Y-1)}{v_Y}+{4(g_Y-1)\log(4 \pi)}
 +\frac{4 p_{\small{\Gamma}}(g_Y-1)}{v_Y}\bigg(\pi+\frac{4\pi^2}{3}+1\bigg).
\end{align}
In our next few steps we show that $\delta_{Y}=o(g_Y \log N) \,\ \text{as}\,\ N \to \infty$. From (\ref{volgeGamma}) and (\ref{ellipticboundGammaN}) it is clear that
\begin{align*}
\frac{8\pi(g_Y-1)}{v_Y}\sum_{j=1}^{e_{\Gamma}} \left( 1+ \frac{1}{\ord(e_j)}\right)+\frac{8(g_Y-1)}{v_Y}\sum_{j=1}^{e_{\Gamma}}(1+\ord(e_j))=O\big(N^\varepsilon\big),
\end{align*}
which are the first two terms of (\ref{deltaX}). 

\vspace{.1cm}\noindent
Then for the third term of \eqref{deltaX} we use \eqref{volgeGamma}, \eqref{dX}, and from \cite{LRS95}, Theorem 1.1, we recall that $
\lambda_1\geq 21\slash 100$. Then we get
\begin{align*}
\frac{8\pi g_Y(g_Y-1)(d_{Y}+1)^2}{\lambda_1v_{\small{\Gamma}}}=O(g_Y).
\end{align*}
 For the fourth and the second last term of \eqref{deltaX}, we use (\ref{volgeGamma}), then we have
\begin{align*}
\frac{8\pi g_Y(g_Y-1)}{v_Y}+{4(g_Y-1)\log(4 \pi)}=O(g_Y).
\end{align*}
For the last term of (\ref{deltaX}), we use (\ref{numberofcusp}), (\ref{volgeGamma}), and we use the well-known identity $\displaystyle\sum\limits_{\substack{d|N\\d>0}} \phi(d)= N$ for the Euler function $\phi$. Then we get
\begin{align*}
\frac{4 p_{\small{\Gamma}}(g_Y-1)}{v_Y}\bigg(\pi+\frac{4\pi^2}{3}+1\bigg)= O(N).
\end{align*}

\vspace{0.2cm} \noindent
Since the estimate (\ref{volgeGamma}) and the formula (\ref{volumeGammaN}) implies that $g_Y= O(N\log N)$, we get 
\begin{align*}
\delta_{Y}=o(g_Y \log N) \,\ \text{as}\,\ N \to \infty.
\end{align*}
For the second last term of \eqref{fromtheorem},
using \eqref{volgeGamma} and \eqref{cX}, we get
\begin{align*}
\frac{8\pi (1-g_Y)c_Y}{v_Y}=o(g_Y \log N) \,\ \text{as}\,\ N \to \infty.
\end{align*}
Then from (\ref{fromtheorem}), as $N\to \infty$, we have
\begin{align}\label{79}
2g_Y(1-g_Y)\, \mathcal{G}_{\mathrm{can}}(0, \infty)&=8 \pi g_Y(1-g_Y)\, \mathcal{C}_{0\infty}+ 4\pi(1-g_Y)\bigg( \sum_{\substack{a\in \mathcal{P}_{\small{\Gamma}} \\ a\not=\infty}}\mathcal{C}_{a \infty}+\sum_{\substack{a\in \mathcal{P}_{\small{\Gamma}} \\ a\not=0}}\mathcal{C}_{a0}\bigg)\notag \\
&+4\pi(1-g_Y)\sum_{j=1}^{e_{\Gamma}}  \left( 1- \frac{1}{\ord(e_j)}\right)\big(\mathcal{K}_{0}(e_j)
+\mathcal{K}_{\infty}(e_j)\big)+ o(g_Y \log N).
\end{align}
From Lemma \ref{Lemma1Gammagen0N}, we have
\begin{align*}
8 \pi g_Y(1-g_Y)\, \mathcal{C}_{0\infty}=2g_Y \log N+o(g_Y \log N) \,\ \text{as}\,\ N \to \infty,
\end{align*}
which is the first term of the right hand side of (\ref{79}). Finally, using Lemma \ref{Lemma2Gammagen0N}, Lemma \ref{Lemma3Gammagen0N}, and Proposition \ref{prop3N}, it is clear that
\begin{align*}
2g_Y(1-g_Y)\, \mathcal{G}_{\mathrm{can}}(0, \infty)=2g_Y \log N+o(g_Y \log N) \,\ \text{as}\,\ N \to \infty,
\end{align*}
which completes the proof.
\end{proof}
\subsection*{Funding} Both authors gratefully acknowledge funding from the LOEWE research unit ``Uniformized structures in Arithmetic and Geometry" at the Technical University of Darmstadt and Goethe University of Frankfurt.

\subsection*{Data Availibility} 
Data sharing is not applicable to this article as no datasets were
generated or analysed during the current study.

\subsection*{Conflict of Interest Statement} Both authors declare that they have no affiliations with or involvement in any organization or entity with any financial or non-financial interest related to the subject matter or materials discussed in this manuscript.

\bibliographystyle{abbrv}
\bibliography{Biblist}

\begin{thebibliography}{10}

\bibitem{abbes}
A.~Abbes and E.~Ullmo.
\newblock Auto-intersection du dualisant relatif des courbes modulaires {$X_0(N)$}.
\newblock {\em J. Reine Angew. Math.}, 484:1--70, 1997.

\bibitem{arakelov}
S.~J. Arakelov.
\newblock An intersection theory for divisors on an arithmetic surface.
\newblock {\em Izv. Akad. Nauk SSSR Ser. Mat.}, 38:1179--1192, 1974.

\bibitem{A12}
A.~Aryasomayajula.
\newblock {\em Bounds for Green’s functions on hyperbolic Riemann surfaces of finite volume}.
\newblock PhD thesis, Humboldt-Universität zu Berlin, 2012.

\bibitem{A15Z}
A.~Aryasomayajula.
\newblock Bounds for {G}reen's functions on noncompact hyperbolic {R}iemann orbisurfaces of finite volume.
\newblock {\em Math. Z.}, 280(1-2):85--133, 2015.

\bibitem{DDC}
D.~Banerjee, D.~Borah, and C.~Chaudhuri.
\newblock Arakelov self-intersection numbers of minimal regular models of modular curves {$X_0(p^2)$}.
\newblock {\em Math. Z.}, 296(3-4):1287--1329, 2020.

\bibitem{Bea}
A.~F. Beardon.
\newblock {\em The geometry of discrete groups}, volume~91 of {\em Graduate Texts in Mathematics}.
\newblock Springer-Verlag, New York, 1995.
\newblock Corrected reprint of the 1983 original.

\bibitem{Ch}
I.~Chavel.
\newblock {\em Eigenvalues in {R}iemannian geometry}, volume 115 of {\em Pure and Applied Mathematics}.
\newblock Academic Press, Inc., Orlando, FL, 1984.
\newblock Including a chapter by Burton Randol, With an appendix by Jozef Dodziuk.

\bibitem{CI00}
J.~B. Conrey and H.~Iwaniec.
\newblock The cubic moment of central values of automorphic {$L$}-functions.
\newblock {\em Ann. of Math. (2)}, 151(3):1175--1216, 2000.

\bibitem{Bruijn}
N.~G. de~Bruijn and J.~H. van Lint.
\newblock On partial sums of {$\sum \sb{d\mid M}\varphi (d)$}.
\newblock {\em Simon Stevin}, 39:18--22, 1965/66.

\bibitem{DI82}
J.-M. Deshouillers and H.~Iwaniec.
\newblock Kloosterman sums and {F}ourier coefficients of cusp forms.
\newblock {\em Invent. Math.}, 70(2):219--288, 1982/83.

\bibitem{DS}
F.~Diamond and J.~Shurman.
\newblock {\em A first course in modular forms}, volume 228 of {\em Graduate Texts in Mathematics}.
\newblock Springer-Verlag, New York, 2005.

\bibitem{DM23}
P.~Dolce and P.~Mercuri.
\newblock Intersection matrices for the minimal regular model of {${X}_0(N)$} and applications to the {A}rakelov canonical sheaf.
\newblock {\em J. Lond. Math. Soc. (2)}, 110(2):Paper No. e12964, 30, 2024.

\bibitem{Grados}
M.~Grados.
\newblock {\em Arithmetic intersections on modular curves}.
\newblock PhD thesis, Humboldt-Universität zu Berlin, 2016.

\bibitem{Grados-Pippich}
M.~Grados and A.-M. von Pippich.
\newblock Self-intersection of the relative dualizing sheaf on modular curves {$X(N)$}, 2022 (arxiv.org/abs/2205.11437).

\bibitem{GZ}
B.~Gross and D.~Zagier.
\newblock Heegner points and derivatives of {$L$}-series.
\newblock {\em Invent. Math.}, 84(2):225--320, 1986.

\bibitem{hej}
D.~A. Hejhal.
\newblock {\em The {S}elberg trace formula for {${\rm PSL}(2,\,{\bf R})$}. {V}ol. 2}, volume 1001 of {\em Lecture Notes in Mathematics}.
\newblock Springer-Verlag, Berlin, 1983.

\bibitem{I02}
H.~Iwaniec.
\newblock {\em Spectral methods of automorphic forms}, volume~53 of {\em Graduate Studies in Mathematics}.
\newblock American Mathematical Society, Providence, RI; Revista Matem\'{a}tica Iberoamericana, Madrid, second edition, 2002.

\bibitem{JK01}
J.~Jorgenson and J.~Kramer.
\newblock Bounds for special values of {S}elberg zeta functions of {R}iemann surfaces.
\newblock {\em J. Reine Angew. Math.}, 541:1--28, 2001.

\bibitem{JK06}
J.~Jorgenson and J.~Kramer.
\newblock Bounds on canonical {G}reen's functions.
\newblock {\em Compos. Math.}, 142(3):679--700, 2006.

\bibitem{JK09}
J.~Jorgenson and J.~Kramer.
\newblock Bounds on {F}altings's delta function through covers.
\newblock {\em Ann. of Math. (2)}, 170(1):1--43, 2009.

\bibitem{JK11}
J.~Jorgenson and J.~Kramer.
\newblock Sup-norm bounds for automorphic forms and {E}isenstein series.
\newblock In {\em Arithmetic geometry and automorphic forms}, volume~19 of {\em Adv. Lect. Math. (ALM)}, pages 407--444. Int. Press, Somerville, MA, 2011.

\bibitem{Kub}
T.~Kubota.
\newblock {\em Elementary theory of {E}isenstein series.}
\newblock Kodansha, Ltd., Tokyo; Halsted Press [John Wiley \& Sons, Inc.], New York-London-Sydney, 1973.

\bibitem{Lang88}
S.~Lang.
\newblock {\em Introduction to {A}rakelov theory}.
\newblock Springer-Verlag, New York, 1988.

\bibitem{LRS95}
W.~Luo, Z.~Rudnick, and P.~Sarnak.
\newblock On {S}elberg's eigenvalue conjecture.
\newblock {\em Geom. Funct. Anal.}, 5(2):387--401, 1995.

\bibitem{PM}
P.~Majumder.
\newblock {\em Bounds for canonical Green’s functions of cofinite Fuchsian groups at cusps}.
\newblock PhD thesis, Technische Universität Darmstadt, 2021.

\bibitem{Mayer}
H.~Mayer.
\newblock Self-intersection of the relative dualizing sheaf on modular curves {$X_1(N)$}.
\newblock {\em J. Th\'{e}or. Nombres Bordeaux}, 26(1):111--161, 2014.

\bibitem{MU}
P.~Michel and E.~Ullmo.
\newblock Points de petite hauteur sur les courbes modulaires {$X_0(N)$}.
\newblock {\em Invent. Math.}, 131(3):645--674, 1998.

\bibitem{Sh}
G.~Shimura.
\newblock {\em Introduction to the arithmetic theory of automorphic functions.}
\newblock Iwanami Shoten Publishers, Tokyo; Princeton University Press, Princeton, N.J., 1971.
\newblock Publications of the Mathematical Society of Japan, No. 11.

\end{thebibliography}
\end{document}